\documentclass[11pt,reqno]{amsart}
\usepackage{amsmath}
\usepackage{amssymb}
\usepackage{amstext}
\usepackage{amsfonts}
\usepackage{amsthm}
\usepackage{ifthen}
\usepackage{xcolor}
\usepackage{xspace}
\usepackage{comment}
\usepackage{graphicx}
\usepackage{mathrsfs}
\usepackage{enumerate}
\usepackage{breakcites}

\numberwithin{equation}{section}  

\newtheorem{punkt}{}[section]

\theoremstyle{plain}

\newtheorem{lemma}[punkt]{Lemma}
\newtheorem{proposition}[punkt]{Proposition}
\newtheorem{theorem}[punkt]{Theorem}

\theoremstyle{definition}

\newtheorem{remarks}[punkt]{Remarks}

\theoremstyle{plain}
\newtheorem*{corollary*}{Corollary}
\newtheorem*{lemma*}{Lemma}
\newtheorem*{proposition*}{Proposition}
\newtheorem*{theorem*}{Theorem}

\theoremstyle{definition}
\newtheorem*{remark*}{Remark}
\newtheorem*{remarks*}{Remarks}
\newtheorem*{example*}{Example}
\newtheorem*{examples*}{Examples}
\newtheorem*{definition*}{Definition}
\newtheorem*{conjecture*}{Conjecture}
\newtheorem*{assumption*}{Assumption}
\newtheorem*{assumptions*}{Assumptions}
\newtheorem*{construction*}{Construction}

\def\mynat{\mathbb{N}}

\def\myreal{\mathbb{R}}

\def\myo{\mathcal{O}}

\def\ee{\mathbb{E}}
\def\pp{\mathbb{P}}

\def\eg{e.g.\@\xspace}
\def\ie{i.e.\@\xspace}
\def\iid{i.i.d.\@\xspace}

\def\wrt{w.r.t.\@\xspace}

\def\eskip{\mskip24mu}

\def\myparagraph#1{\paragraph{\textbf{#1.}}}

\begin{document}

\title{The Entropic Erd{\H{o}}s-Kac Limit Theorem}

\author{S.~G.~Bobkov$^{1,4}$, G.~P.~Chistyakov$^{2,4}$, and H.~K\"osters$^{3,4}$}%

% \date{August 14th, 2013}
% \filename{eclt-48.tex}

\begin{abstract}
We prove entropic and total variation versions of the Erd{\H{o}}s-Kac limit theorem
for the maximum of the partial sums of \iid random variables with densities.
\end{abstract}

\maketitle

\footnotetext{1) School of Mathematics, University of Minnesota, USA; \\ Email: bobkov@math.umn.edu.}
\footnotetext{2) Department of Mathematics, Bielefeld University, Germany; \\ Email: chistyak@math.uni-bielefeld.de.}
\footnotetext{3) Department of Mathematics, Bielefeld University, Germany; \\ Email: hkoesters@math.uni-bielefeld.de.}
\footnotetext{4) Research supported by CRC 701. \\ S.B. was also supported by DMS-1106530 and Simons Fellowship.}

\markboth{Entropic Limit Theorem}{Entropic Limit Theorem}

\section{Introduction}

Let $\{X_n\}_{n \geq 1}$ be independent identically distributed (\iid) 
random variables with mean $\ee X_1 = 0$ and variance $\ee X_1^2 = 1$. 
Put
$$
S_n := \sum_{k=1}^{n} X_k, \qquad
\overline{S}_n := \max_{k=1,\hdots,n} S_k, \qquad n \in \mynat.
$$
Throughout  we denote by $Z$ a standard normal random variable
with its density $\varphi(x) := \tfrac{1}{\sqrt{2\pi}}\, e^{-x^2/2}$
and use the symbol $\Rightarrow$ to denote convergence in distribution.

The classical central limit theorem states that
\begin{align}
\label{eq:clt}
S_n/\sqrt{n} \Rightarrow Z \qquad \text{as $n \to \infty$.}
\end{align}
In 1986 Barron \cite{Ba} established an~entropic~version of this result,
the so-called \emph{entropic central limit theorem}.
To formulate it, first let us introduce some notation.
Let $Y$ be a random variable with density $\psi$, and let $X$ be 
a random variable whose distribution is absolutely continuous with 
respect to that of $Y$. The relative entropy of $X$ with respect 
to $Y$ is~defined~by
\begin{align}
\label{eq:defent}
D(X \,|\, Y) := \int_{\{\psi(x) > 0\}} 
L\left(\frac{p(x)}{\psi(x)}\right) \psi(x) \, dx \,,
\end{align}
where $p$ is a density of $X$,
$L(x) := x \log x$ for $x > 0$ and $L(x) := 0$ for $x = 0$.
In~case the distribution of $X$ is not absolutely continuous
with respect to that of~$Y$, put $D(X \,|\, Y) := \infty$.
Then the entropic central limit theorem by Barron states that
\begin{align}
\label{eq:eclt}
D(S_n/\sqrt{n} \,|\, Z) \to 0 \qquad \text{as $n \to \infty$}
\end{align}
if and only if $D(S_{n_0} \,|\, Z) < \infty$ for some $n_0 \in \mynat$.
\pagebreak[1]
This result is motivated, inter alia, by the distinguished property
of the standard normal distribution that it~maximizes (Shannon) entropy.
%See further below for details.

Barron's result has sparked much further research on entropic limit 
theorems. For instance, there are several publications devoted to the 
rate of convergence, see 
Artstein, Ball, Barthe and Naor \cite{ABBN:2004b},
Johnson and Barron \cite{Johnson-Barron:2004}, 
Johnson \cite{Johnson:2004}, and 
Bobkov, Chistyakov and G\"otze \cite{BCG:2011b}.
Entropic limit theorems have also been derived for certain non-normal 
limit distributions within the class of stable laws, cf.
\cite{Johnson:2004}, \cite{BCG:2011c}, \cite{KHJ:2005}.

All these limit distributions arise in connection with sums of 
\iid random variables and are therefore infinitely divisible.
Our aim is to investigate a different situation,
namely for the maxima of sums of \iid summands,
with a limit distribution that is not infinitely divisible.
\linebreak[2]
Here the analogue of the classical central~limit theorem
is given by the Erd{\H{o}}s-Kac limit theorem \cite{EK},
which states that
\begin{align}
\label{thm:erdos-kac}
\overline{S}_n / \sqrt{n} \Rightarrow |Z| \qquad \text{as $n \to \infty$.}
\end{align}
The distribution of $|Z|$, which has density
$\varphi_+(x) := \sqrt{\frac{2}{\pi}}\, e^{-x^2/2}\,\pmb{1}_{(0,\infty)}(x)$,
is com\-mon\-ly called the one-sided (or reflected) standard normal law.
As explained below, this distribution plays a similar role 
to the normal distribution in that it~maximizes entropy 
among all \emph{positive} random variables with fixed second moment. 
It is therefore quite natural to ask whether the Erd{\H{o}}s-Kac limit 
theorem \cite{EK} also admits an~entropic formulation.

To state a corresponding assertion, we introduce more notation.
Given a random variable $X$ such that $\pp(X > 0) > 0$, 
let $\widetilde{X}$ have the same distribution as $X$ conditioned to be positive, 
\ie $\pp(\widetilde{X} \in A) = \pp(X \in A| X>0)$ for Borel sets $A$ on the real line.
Then the relative entropy of $X$ conditioned to be positive 
with respect to a positive random variable $Y$ with density $\psi$ 
is defined by
\begin{align}
\label{eq:defentplus}
D_+(X \,|\, Y) := D(\widetilde{X} \,|\, Y). 
\end{align}
In the sequel, $Y$ will always be given by $|Z|$ or some scalar multiple 
of it. Our main result is as~follows:

\begin{theorem}
\label{thm:ECLT}
Suppose that $X_1,X_2,\hdots$ are \iid random variables
with a density, mean zero and variance one. Then
\begin{align}
\label{eq:ECLT-1}
D_+(\overline{S}_n/\sqrt{n}\,|\,|Z|) \to 0 
\qquad \text{as \ $n \to \infty$}
\end{align}
if and only if
\begin{align}
\label{eq:ECLT-2}
D_+(X_1 \,|\, |Z|) < \infty \,.
\end{align}
\end{theorem}

%\begin{remark*}
In fact, the assumption that the $X_j$ have a density is only for 
convenience and could be omitted. Note, however, that 
$\eqref{eq:ECLT-2}$ implies that $X_1$ has a density on the positive 
half-line.
%\end{remark*}

\pagebreak[2]
Let us recall that the relative entropy represents
a rather strong measure of deviation of distributions.
Indeed, by the Pinsker-Csisz\'ar-Kullback inequality, 
\begin{align}
\label{eq:pinsker}
D(X \,|\, Z) \geq \tfrac12 \, (d_{TV}(X,Z))^2 \,,
\end{align}
where $d_{TV}(X,Z)$ denotes the total variation distance
between the distributions of $X$ and $Z$
(cf. \cite{Pinsker:1964,Csiszar:1967,Kullback:1967,Fedotov-Harremoes-Topsoe:2003}).
Thus, \eqref{eq:eclt} implies $d_{TV}(S_n/\sqrt{n},Z) \to 0$
as~$n \to \infty$, and hence \eqref{eq:clt}.
Similarly, \eqref{eq:ECLT-1} implies 
$d_{TV}(\overline{S}_n/\sqrt{n},|Z|) \to 0$, 
and hence \eqref{thm:erdos-kac}.
This follows from \eqref{eq:pinsker} in combination with
the well-known fact that, under our moment assumptions,
\begin{align}
\label{eq:osn}
\pp(\overline{S}_n \leq 0) = \myo(n^{-1/2})
\end{align}
(cf. \eg \cite[pp.\,414f]{Feller-2}).
In fact, for convergence in total variation distance, condition \eqref{eq:ECLT-2}
is not needed, since we have:

\begin{theorem}
\label{thm:totalvariation}
Suppose that $X_1,X_2,\hdots$ are \iid random variables with 
a density, mean zero and variance one. Then
$$
d_{TV}(\overline{S}_n / \sqrt{n}, |Z|) \to 0 
\qquad \text{as \ $n \to \infty$.}
$$
\end{theorem}

\medskip
As already mentioned, both the centered and the one-sided normal 
distribution play a special role from the viewpoint of information 
theory. Let us recall that for a random variable $X$ with density $p$, 
the entropy (also called Shannon entropy or differential entropy) 
is defined by
$$
h(X) := - \int_{-\infty}^{\infty} L(p(x)) \, dx \,,
$$
where $L$ is as in \eqref{eq:defent}.
%(If $X$ does not have a density, put~$h(X) := -\infty$.)
If $\ee X^2 = \sigma^2$ is finite, then the entropy is well-defined, and
$$
h(\sigma Z) - h(X) = D(X \,|\, \sigma Z) \geq 0 \, \qquad
\big(Z \sim N(0,1)\big),
$$
with equality if and only if $X$ and $\sigma Z$ have the same 
distribution. Thus, the centered normal distribution with second 
moment $\sigma^2$ maximizes entropy among all probability measures 
with the same second moment. Moreover, in
$$
D(X \,|\, \tau Z) = 
- h(X) + \tfrac12 \log(2\pi\tau^2) + \tfrac12 \sigma^2 / \tau^2 
\qquad (\tau > 0)
$$
the right-hand side is minimized for $\tau = \sigma$, so 
$D(X \,|\, \sigma Z)$ may be interpreted as a measure of deviation 
of the distribution of $X$ from the class of all centered normal 
distributions. Similarly, for a positive random variable $X$
with finite second~moment $\ee X^2 = \sigma^2$,
$$
h(\sigma |Z|) - h(X) = D_+(X \,|\, \sigma |Z|) \geq 0 \,,
$$
with equality if and only if $X$ and $\sigma |Z|$ have the same 
distribution. Hence, the one-sided normal distribution with second 
moment $\sigma^2$ maximizes entropy among all probability measures 
on the positive half-line with the same second moment.
Also, in
$$
D_+(X \,|\, \tau |Z|) = 
- h(X) + \tfrac12 \log (\tfrac12\pi\tau^2) + \tfrac12 \sigma^2/\tau^2
\qquad (\tau > 0)
$$
the right-hand side is minimized for $\tau = \sigma$. Therefore, 
as above, $D_+(X \,|\, \sigma |Z|)$ may be interpreted as a measure 
of deviation of the distribution of~$X$
from the class of all one-sided normal distributions.

In this respect, note that
\begin{align}
\label{eq:second-moment}
\ee(\overline{S}_n^+/\sqrt{n})^2 = 1 + o(1) \qquad \text{as $n \to \infty$,}
\end{align}
see \eg Section~6 below.
Combining \eqref{eq:osn} and \eqref{eq:second-moment}, 
it is easy to see that for~large~$n$,
$\overline{S}_n / \sqrt{n}$ conditioned to~be positive 
has second moment approximately equal to $1$, 
so that the comparison to $|Z|$ in~\eqref{eq:ECLT-1} 
is natural.

\pagebreak[3]

Finally, let us emphasize the following curious difference between 
the entropic central limit theorem and our Theorem
\ref{thm:ECLT}. Even if $X_1$ itself has density, Barron's 
characterization uses the finiteness of $D(S_{n_0} \,|\, Z)$
for some $n_0 \in \mynat$ (which may be any natural number);
see \cite{Ba} for an example requiring $n_0 > 1$.
In contrast to that, our characterization uses $n_0 = 1$ at once.
More precisely, it~follows from our proof that
$D_+(\overline{S}_{n_0} \,|\, |Z|) < \infty$ for some ${n_0} \in \mynat$
if and only if this is true for $n_0 = 1$.
% The necessity of this~condition is a simple consequence of the fact that 
% the maximum $\overline{S}_{n^{}_0}$ coincides with $\overline{S}_{1}$ 
% with positive probability.

In the proof of (1.3) given in \cite{Ba},
entropy convolution inequalities for sums of independent random variables 
play a major role. In our analysis for the maxima of sums,
these inequalities still play a role in the proof of Theorem~\ref{thm:ECLT}, 
but they have less far-reaching consequences.
To control the density of the maximum,
we~use more classical methods and results based on Fourier analysis,
see Nagaev \cite{Nagaev:69,Nagaev:70} and Aleshkyavichene \cite{Alesh:77}.
This approach does not only lead to proofs of entropic limit theorems
(cf.\@\xspace \cite{BCG:2011c}),
but in principle, similarly as in \cite{BCG:2011b},
it~should also lead to results on the (exact) rate of convergence.
Apparently, such refined results cannot be obtained
by using known information-theoretic tools.

A major ingredient in our proof will be the local limit theorem 
for maxima of sums of \iid random variables from \cite{Alesh:77},
see~also \cite{Alesh:73,Nagaev-Eppel:76,Wachtel:2012} for related results.
To~obtain \eqref{eq:ECLT-1} under minimal conditions,
we~need to extend the result from \cite{Alesh:77}
from bounded to unbounded densities
(see Proposition \ref{prop:local}).

Let us introduce some conventions for the rest of the paper.
We assume that the random variables $X_j$ are \iid and have 
a density, mean $0$ and variance $1$.
Unless otherwise indicated, we write $p$ for their density,
$F$ for their distribution function and $f$ for their characteristic 
function. Moreover, let $p_n,F_n,f_n$ and
$\overline{p}_n,\overline{F}_n,\overline{f}_n$
denote the corresponding functions for the random variables $S_n$ and 
$\overline{S}_n$.
%(Note that the densities exist due 
%to our assumption that the $X_j$ have a~density.)
We write $p_n^*$ and $\overline{p}_n^*$ for the densities of 
the rescaled random variables 
$S_n/\sqrt{n}$ and $\overline{S}_n/\sqrt{n}$.

For a real number $x$, 
set $x^+ := \max\{x,0\}$ and $x^- := \max\{-x,0\}$. 
Unless otherwise indicated, $\myo$-bounds and $o$-bounds
refer to the case where $n \to \infty$ and hold uniformly 
in $x$ (in the region under consideration).
Finally, $C_1,C_2,\dots$ denote positive~constants
which may depend on the dis\-tribution of the $X_j$
and which may~change from step to step.

The paper is organized as~follows.
Section~2 contains some preliminary remarks on relative entropy.
Section~3--7 are devoted to the proof 
of the sufficiency part of Theorem \ref{thm:ECLT},
while the necessity part of Theorem \ref{thm:ECLT}
is proved in Section~8.
Section~9 contains the proof of Theorem \ref{thm:totalvariation}.

\pagebreak[2]
\medskip
\section{Some Remarks on Relative Entropy}

Throughout this section, let $\psi$ be a positive 
probability density on the positive half-line.
Given a non-negative measurable function $f$ on the real line,
set 
\begin{align}
\label{eq:relentdef}
D(f \,|\, \psi) := \int_0^\infty L\bigg(\frac{f(x)}{\psi(x)}\bigg) \, \psi(x) \, dx \,,
\end{align}
where $L(x)$ is the function defined in the introduction.
By abuse of terminology, we~will call $D(f \,|\, \psi)$ 
\emph{relative entropy} 
even when $f$ is not a probability density on the positive half-line.
Note that in this special case, we have $D(f \,|\, \psi) \geq 0$ 
by Jensen's inequality. \linebreak[2]
If $f$ is an arbitrary non-negative measurable function, this need not be~true anymore, 
but we have at~least $D(f \,|\, \psi) \geq \min \{ L(x) : x \geq 0 \} = -e^{-1}$.

Let us collect some basic properties of relative entropy which will be used later. 
(Some of the proofs are straightforward, which is why we omit them.)

\begin{lemma}
\label{lemma:scalar}
Suppose that $\alpha$ is a positive real number 
and $f$ is a non-negative measurable function
with $\int_0^\infty f(x) \, dx < \infty$.
Then
\[
D(\alpha f \,|\, \psi) = \alpha D(f \,|\, \psi) + L(\alpha) \int_0^\infty f(x) \, dx \,.
\]
\end{lemma}

\begin{lemma}
\label{lemma:combination}
Suppose that $\alpha_1,\hdots,\alpha_n$ are positive real numbers
and $f_1,\hdots,f_n$ are non-negative measurable functions
with $\int_0^\infty f_k(x) \, dx < \infty$, $k=1,\hdots,n$.
Then
\[
D\left(\left.\sum_{k=1}^{n} \alpha_k f_k \,\right|\, \psi\right) 
\leq 
\sum_{k=1}^{n} \alpha_k D(f_k \,|\, \psi) 
+
\left(\log \sum_{k=1}^{n} \alpha_k\right) \, \sum_{k=1}^{n} \alpha_k \int_0^\infty f_k(x) \, dx \,.
\]
\end{lemma}

\pagebreak[2]

\begin{lemma}
\label{lemma:posneg}
Suppose that $\psi$ is decreasing on the positive half-line
and that $f$~and~$g$ are probability densities 
on $(0,+\infty)$ and $(-\infty,0)$, respectively.
Then
$$
D(f \ast g \,|\, \psi) \leq D(f \,|\, \psi) + e^{-1} \,.
$$
\end{lemma}

\begin{proof}[Proof of Lemma \ref{lemma:posneg}] 
Since $L$ is a convex~function and $g$ is a probability density
on $(-\infty,0)$, it follows from Jensen's inequality that
$$
L\left(\int_{-\infty}^{0} h(y) \, g(y) \, dy \right) 
\leq
\int_{-\infty}^{0} L(h(y)) \, g(y) \, dy 
$$
for any non-negative measurable function $h$.
We therefore obtain
\begin{align*}
   D(f \ast g | \psi)
&= \int_0^\infty L\left( \int_{-\infty}^{0} \frac{f(x-y)}{\psi(x)} \, g(y) \, dy \right) \, \psi(x) \, dx \\
&\leq \int_0^\infty \int_{-\infty}^{0} L\left( \frac{f(x-y)}{\psi(x)} \right) \, g(y) \, dy \, \psi(x) \, dx \\
&= \int_{-\infty}^{0} \int_0^\infty f(x-y) \log\left( \frac{f(x-y)}{\psi(x)} \right) \, dx \, g(y) \, dy \,.
\end{align*}
Since $\psi(x)$ is decreasing in $x$, we have, for any $y < 0$,
\begin{multline*}
  \int_0^\infty f(x-y) \log\left( \frac{f(x-y)}{\psi(x)} \right) \, dx
\leq \int_0^\infty f(x-y) \log\left( \frac{f(x-y)}{\psi(x-y)} \right) \, dx  
\\
= \int_{0}^\infty L\left( \frac{f(u)}{\psi(u)} \right) \, \psi(u) \, du - \int_{0}^{-y} L\left( \frac{f(u)}{\psi(u)} \right) \, \psi(u) \, du
\leq D(f | \psi) + e^{-1} \,.
\end{multline*}
Combining these estimates, we get
$$
D(f \ast g \,|\, \psi) \leq D(f \,|\, \psi) + e^{-1} \,,
$$
and the lemma is proved.
\end{proof}

\begin{lemma}
\label{lemma:convexity}
Suppose that $f$ and $g$ are non-negative measurable functions
with $\alpha := \int_0^\infty f(x) \, dx < \infty$
and $\beta := \int_0^\infty g(x) \, dx < \infty$.
Then
$$
D(f \,|\, \psi) + D(g \,|\, \psi) \leq D(f + g \,|\, \psi) \leq D(f \,|\, \psi) + D(g \,|\, \psi) + L(\alpha\!+\!\beta) - L(\alpha) - L(\beta) \,.
$$
\end{lemma}

\begin{proof}
Suppose w.l.o.g. that $\alpha,\beta > 0$.
On the one hand, by Lemmas \ref{lemma:scalar}~and~\ref{lemma:combination}, we~have
\begin{align*}
   D(f+g | \psi)
&= (\alpha\!+\!\beta) \, D(\tfrac{\alpha}{\alpha+\beta} \tfrac{f}{\alpha}+\tfrac{\beta}{\alpha+\beta} \tfrac{g}{\beta} | \psi) + L(\alpha+\beta) \\
&\leq (\alpha\!+\!\beta) \left[ \tfrac{\alpha}{\alpha+\beta} D(\tfrac{f}{\alpha}|\psi) + \tfrac{\beta}{\alpha+\beta} D(\tfrac{g}{\beta}|\psi) \right] + L(\alpha+\beta) \\
&= \alpha D(\tfrac{f}{\alpha}|\psi) + \beta D(\tfrac{g}{\beta}|\psi) + L(\alpha+\beta) \\
&= D(f|\psi) + D(g|\psi) - L(\alpha) - L(\beta) + L(\alpha+\beta) \,.
\end{align*}
On the other hand, it is straightforward to check that
$L(x+y) \geq L(x) + L(y)$ for~any $x,y \geq 0$, whence
$
D(f+g | \psi) \geq D(f | \psi) + D(g | \psi).
$
\end{proof}

\pagebreak[2]

In particular, it follows from Lemmas \ref{lemma:scalar} and \ref{lemma:convexity} 
that for any non-negative measurable functions $f,g$ with
$\int_0^\infty f(x) \, dx < \infty$, $\int_0^\infty g(x) \, dx < \infty$
and any $\alpha,\beta > 0$, we~have
\begin{align}
\label{eq:finiteness}
D(\alpha f+\beta g \,|\, \psi) < \infty
\quad\text{if and only if}\quad
D(f \,|\, \psi) < \infty \text{ and } D(g \,|\, \psi) < \infty \,.
\end{align}

\pagebreak[2]

\begin{lemma}
\label{lemma:perturb}
Suppose that $(f_n)$ and $(g_n)$ are sequences of non-negative measurable functions
such that 
$\int_0^\infty f_n(x) \, dx = 1 + o(1)$
and
$\int_0^\infty g_n(x) \, dx = o(1)$
as $n \to \infty$. Then 
$$
D(f_n+g_n \,|\, \psi) = D(f_n \,|\, \psi) + D(g_n \,|\, \psi) + o(1) \qquad \text{as $n \to \infty$} \,.
$$
\end{lemma}

\begin{proof}
This is an immediate consequence of Lemma \ref{lemma:convexity}.
\end{proof}

In the following sections, $\psi$ will always be given by the probability density
$\varphi_+(x) := \sqrt{\frac{2}{\pi}} \, e^{-x^2/2}$ (\mbox{$x > 0$})
or its rescaled version
$\varphi_{n,+}(x) := \sqrt{\frac{2}{\pi n}} \, e^{-x^2/2n}$ (\mbox{$x > 0$}),
where $n \in \{ 1,2,3,\hdots \}$. 
Note that $\varphi_{n,+}$ is the density 
of the one-sided normal distribution with second moment $n$.
It is easy to check that for any~non-negative measurable function $f$, we have
\begin{align}
\label{eq:scalinginvariance}
D(\sqrt{n} f(\sqrt{n} \,\cdot\,) \,|\, \varphi_+) = D(f \,|\, \varphi_{n,+}) \,.
\end{align}

\pagebreak[2]
\medskip
\section{Binomial Decomposition}

In this section we start with the proof of sufficiency
in Theorem \ref{thm:ECLT}.
In~the~sequel, by a signed density we mean any measurable function
$h(x)$ defined on the real line or on the positive half-line
such that
$
\int_{-\infty}^\infty |h(x)|\,dx < \infty.
$
Since it is more convenient to work with bounded densities,
we use a \emph{binomial decomposition} of the density $p$
to write the density $\overline{p}_n^*$ 
(restricted to the positive half-line)
as the sum of two signed~densities,
a bounded term $\overline{q}_n^*$ and a remainder term $\overline{r}_n^*$.
This~representation will play an important role in the proof 
\pagebreak[2] of the sufficiency part of Theorem \ref{thm:ECLT}.
Let us remark that binomial decompositions are a well-known tool 
in the investigation of the classical central limit theorem, 
see \eg \cite{Sirazdinov-Mamatov:1962,Ibragimov-Linnik:1965}.
In~connection with entropic central limit theorems,
they have recently been used in \cite{BCG:2011b,BCG:2011c}.

\medskip

Recall that $p$ is the density of $X_1$.
Write 
\begin{align}
\label{eq:decomposition-p0}
p = (1-\varrho) q_1 + \varrho q_2 \,,
\end{align}
where $q_1$ is a bounded probability density
with $\int_0^\infty q_1(x) \, dx > 0$,
$q_2$ is a potentially unbounded probability density,
and $0 \leq \varrho < \tfrac12$.
It follows that for any $n \geq 1$,
\begin{multline}
  p_n(x) 
= p^{\ast n}(x) 
= \left( \sum_{k=1}^{n} \tbinom{n}{k} (1-\varrho)^k \varrho^{n-k} \, (q_1^{\ast k} \ast q_2^{*(n-k)})(x) \right) + \varrho^n q_2^{*n}(x) \\
=: (1-\varrho^n) q_{n,1}(x) + \varrho^n q_{n,2}(x) \,,
\label{eq:decomposition-p}
\end{multline}
where $q_{n,1}(x)$ and $q_{n,2}(x)$
are again probability densities.
% We assume from now on that $\varrho > 0$,
% the case $\varrho = 0$ being simpler.

We now need the following formula due to Nagaev \cite[Equation (0.8)]{Nagaev:70a}:
\linebreak For $n \in \mynat$ and $t \in \myreal$, we~have
\begin{align}
\label{eq:nagaev}
\ee e^{it\overline{S}_n} = \sum_{k=1}^{n} f^k(t) \overline\varphi_{n-k}(t) \,,
\end{align}
where 
\begin{align}
\label{eq:varphidef}
\overline\varphi_0(t) := 1 \qquad \text{and} \qquad \overline\varphi_k(t) := \int_{-\infty}^{0} (1-e^{itx}) \, d\overline{F}_{k}(x) \qquad (k > 0) \,.
\end{align}
By \eqref{eq:nagaev} and the uniqueness theorem for Fourier transforms (of signed measures), 
it follows that the density of $\overline{S}_n := \max \{ S_1,\hdots,S_n \}$
is given by
$$
\overline{p}_n(x) = \sum_{k=1}^{n} (p^{\ast k} \ast G_{n-k})(x) \,,
$$
where 
$$
G_{0}(dx) := \delta_0(dx), 
\quad
G_{k}(dx) := \overline{F}_{k}(0) \delta_0(dx) - \overline{p}_{k}(x) \, \pmb{1}_{(-\infty,0)}(x) \, dx 
\quad 
\text{for $k > 0$}
$$
and $(p^{\ast k} \ast G_{n-k})(x) := \int p^{\ast k}(x-y) \, G_{n-k}(dy)$.

\pagebreak[2]

Using \eqref{eq:decomposition-p}, we may write
\begin{align}
\label{eq:decomposition-op}
\overline{p}_n(x) = \overline{q}_n(x) + \overline{r}_n(x) \,,
\end{align}
where
\begin{align}
\label{eq:qrdef}
\overline{q}_n(x) := \sum_{k=1}^{n} (1-\varrho^k) (q_{k,1} \ast G_{n-k})(x) \,,
\quad
\overline{r}_n(x) := \sum_{k=1}^{n} \varrho^k (q_{k,2} \ast G_{n-k})(x) \,.
\end{align}
Note that each $\overline{q}_n$ is bounded,
since the $q_{k,1}$ are bounded and the $G_{n-k}$ are finite signed measures.
\pagebreak[2]
The main idea is to use $\overline{q}_n$
as a~bounded approximation to~$\overline{p}_n$.
Of~course, $\overline{q}_n$ and $\overline{r}_n$ 
are only signed densities in general. However, 
they may be represented as differences of non-negative densities 
by writing
$$
\overline{q}_n(x) = \overline{q}_n^+(x) - \overline{q}_n^-(x) 
\quad\text{and}\quad
\overline{r}_n(x) = \overline{r}_{n,1}(x) - \overline{r}_{n,2}(x) \,,
$$
where $\overline{q}_n^+$ and $\overline{q}_n^-$
denote the positive and negative part of $\overline{q}_n$
and $\overline{r}_{n,1}$ and $\overline{r}_{n,2}$
are~defined~by
$$
\overline{r}_{n,j}(x) := \sum_{k=1}^{n} \varrho^k (q_{k,2} \ast G^{\pm}_{n-k})(x)
$$
($j=1,2$), where $\pm = +$ for $j=1$, $\pm = -$ for $j=2$,
and $G^+_{n-k}$ and $G^-_{n-k}$ denote \linebreak 
the positive and negative part of the signed measure $G_{n-k}$.
Note that $\overline{r}_{n,1}$ and $\overline{r}_{n,2}$
are \emph{not} the positive and negative part of $\overline{r}_n$
in general. 

\pagebreak[2]

Thus, we obtain
\begin{align}
\label{eq:key-1}
\overline{p}_n = (\overline{q}_n^+ - \overline{q}_n^-) + (\overline{r}_{n,1} - \overline{r}_{n,2})
\end{align}
or (equivalently)
\begin{align}
\label{eq:key-2}
\overline{p}_n + \overline{q}_n^- + \overline{r}_{n,2} = \overline{q}_n^+ + \overline{r}_{n,1} \,.
\end{align}
Write
$\overline{p}_n^*(x) := \sqrt{n} \, \overline{p}_n(\sqrt{n} \, x)$,
$\overline{q}_n^*(x) := \sqrt{n} \, \overline{q}_n(\sqrt{n} \, x)$,
$\overline{r}_n^*(x) := \sqrt{n} \, \overline{r}_n(\sqrt{n} \, x)$,
etc. for the rescaled versions of the above densities.
We then have the following result.

\begin{lemma}
\label{lemma:entropy-q} \ 
\begin{enumerate}[(a)]
\item
$\int_0^\infty |\overline{p}_n^*(x) - \overline{q}_n^*(x)| \, dx = \myo(n^{-1/2}) \,.$
\item
$\int_0^\infty x^2 |\overline{p}_n^*(x) - \overline{q}_n^*(x)| \, dx = \myo(n^{-1/2}) \,.$
\item
If \eqref{eq:ECLT-2} holds then $D(\overline{p}_n^* \,|\, \varphi_+) = D((\overline{q}_n^*)^+ \,|\, \varphi_+) + o(1)$ as $n \to \infty$.
\end{enumerate}
\end{lemma}

\begin{proof}
Throughout this proof, for any measurable function $p$, we~write
$$
\|p\|_1 := \int_0^\infty |p(x)| \, dx
$$
for the \emph{total variation norm} (of the associated signed measure) and
$$
\|p\|_\infty := \sup_{x \in (0,\infty)} |p(x)|
$$
for the \emph{supremum norm}. Furthermore, if $p$ is non-negative,
we write $D(p \,|\, \varphi_+)$ for the relative entropy
as in \eqref{eq:relentdef}. 
Recall the probability densities $\varphi_{n,+}$
introduced at the end of Section~2.

\pagebreak[2]
\medskip
\myparagraph{Analysis of $\pmb{\overline{r}_{n,j}^*(x)}$} 
By \eqref{eq:osn}, $\overline{F}_{n}(0) = \myo(n^{-1/2})$ as $n \to \infty$.
Thus,
\begin{align}
\label{eq:tvr}
\|\overline{r}_{n,j}^*\|_1 = \|\overline{r}_{n,j}\|_1 \leq \sum_{k=1}^{n} \overline{F}_{n-k}(0) \varrho^k \leq \sum_{k=1}^{n} \frac{C_1 \varrho^k}{\sqrt{n-k+1}} = \myo (n^{-1/2}) \,,
\end{align}
$j=1,2$. Also, since $G_{n-k}^{\pm}$ is concentrated on $(-\infty,0]$,
$$
\int_0^\infty x^2 \overline{r}_{n,j}^*(x) \, dx 
\leq \tfrac1n \sum_{k=1}^{n} \frac{C_1 \varrho^k}{\sqrt{n-k+1}} \int_{-\infty}^{\infty} x^2 q_{k,2}(x) \, dx \,,
$$
$j=1,2$. Let $Y_1,\hdots,Y_k$ be \iid random variables with density $q_2$.
Then
$$
  \int_{-\infty}^{\infty} x^2 q_{k,2}(x) \, dx
= \|Y_1+\hdots+Y_k\|_2^2 
\leq k^2 \, \|Y_1\|_2^2 \,,
$$
and we come to the conclusion that
\begin{align}
\label{eq:smr}
\int_0^\infty x^2 \overline{r}_{n,j}^*(x) \, dx \leq \tfrac1n \sum_{k=1}^{n} \frac{C_1 \varrho^k}{\sqrt{n-k+1}} \int_{-\infty}^{\infty} x^2 q_{k,2}(x) \, dx = \myo (n^{-3/2}) \,,
\end{align}
$j=1,2$. Clearly, \eqref{eq:tvr} and \eqref{eq:smr} imply (a) and (b).

We will now show that if \eqref{eq:ECLT-2} holds then
\begin{align}
\label{eq:entropy-r}
  D\left(\overline{r}_{n,j}^* \,|\, \varphi_+\right) 
= D\left(\left. \sum_{k=1}^{n} \varrho^k \, q_{k,2} \ast G_{n-k}^{\pm} \right| \varphi_{n,+} \right)
= o(1) \,,
\end{align}
$j=1,2$. 
We provide the details for $\overline{r}_{n,2}^*$ only,
the argument for $\overline{r}_{n,1}^*$ being similar.
% but~simpler.

Note that $G_{0}^{-} = 0$. For $k = 1,\hdots,n-1$,
write $G_{n-k}^{-}(dx) = \overline{F}_{n-k}(0) \, s_{n-k}(x) \, dx$,
where $s_{n-k}(x) := \overline{p}_{n-k}(x) / \overline{F}_{n-k}(0)$ 
($x < 0$) is a probability density on $(-\infty,0)$.
Also, write $q_2 = \lambda_+ q_{2,+} + \lambda_- q_{2,-}$,
where $\lambda_+,\lambda_- \geq 0$, $\lambda_+ + \lambda_- = 1$,
and $q_{2,+}$ and $q_{2,-}$ \linebreak[2] are probability densities 
on $(0,+\infty)$ and $(-\infty,0)$, respectively.
Then
$$
q_{k,2} = \sum_{j=0}^{k} \tbinom{k}{j} \lambda_+^j \, \lambda_-^{k-j} \, q_{2,+}^{\ast j} \ast q_{2,-}^{\ast (k-j)} \,,
$$
and it follows by a two-fold application of Lemma \ref{lemma:combination} that
\begin{align*}
D&\left(\left.\sum_{k=1}^{n} \varrho^k \, q_{k,2} \ast G_{n-k}^{-} \right| \varphi_{n,+}\right) \\
&\leq \sum_{k=1}^{n-1} \varrho^k \, \overline{F}_{n-k}(0) \, D\left( q_{k,2} \ast s_{n-k} \,\Big|\, \varphi_{n,+} \right) + \myo(\log n / \sqrt{n}) \\
&\leq \sum_{k=1}^{n-1} \varrho^k \, \overline{F}_{n-k}(0) \, \sum_{j=1}^{k} \tbinom{k}{j} \lambda_+^j \, \lambda_-^{k-j} \, D\left(\left. q_{2,+}^{\ast j} \ast q_{2,-}^{\ast (k-j)} \ast s_{n-k} \right| \varphi_{n,+}\right) + \myo(\log n / \sqrt{n}) \,.
\end{align*}
(For the last step, note that $D( q_{2,-}^{\ast k} \ast s_{n-k} \,|\, \varphi_{n,+}) = 0$.)
Using Lemma \ref{lemma:posneg} 
with $f(x) := q_{2,+}^{\ast j}(x)$
and $g(x) := (q_{2,-}^{\ast (k-j)} \ast s_{n-k})(x)$,
we~get
$$
     D\left(\left. q_{2,+}^{\ast j} \ast q_{2,-}^{\ast (k-j)} \ast s_{n-k} \right| \varphi_{n,+}\right)
\leq D\left(q_{2,+}^{\ast j} \,\Big|\, \varphi_{n,+}\right) + e^{-1} \,.
$$
Let $\mu$ and $\sigma^2$ denote the mean and variance 
of the probability density $q_{2,+}$,
and let $\varphi_{\mu,\sigma^2}$ denote the density
of the Gaussian distribution with mean $\mu$ and variance $\sigma^2$.
As a consequence of the entropy power inequality 
(see \eg Theorem~4 in \cite{DCT:91}), we have
$$
D(q_{2,+}^{\ast j} | \varphi_{j\mu,j\sigma^2}) \leq D(q_{2,+} | \varphi_{\mu,\sigma^2}) \,, \quad j \geq 1 \,.
$$
We therefore obtain
\begin{align*}
   D(q_{2,+}^{\ast j} | \varphi_{n,+})
&= \int_0^\infty q_{2,+}^{\ast j} \log \left( \frac{q_{2,+}^{\ast j}}{\varphi_{j\mu,j\sigma^2}} \frac{\varphi_{j\mu,j\sigma^2}}{\varphi_{n,+}} \right) \, dx \\
&= D(q_{2,+}^{\ast j} | \varphi_{j\mu,j\sigma^2}) + \int_0^\infty q_{2,+}^{\ast j} \log \left( \frac{\varphi_{j\mu,j\sigma^2}}{\varphi_{n,+}} \right) \, dx \\
% &= D(q_{2,+}^{\ast j} | \varphi_{j\mu,j\sigma^2}) + \int_0^\infty q_{2,+}^{\ast j}(x) \left( -\tfrac12\log(2\pi j\sigma^2) - \frac{(x-j\mu)^2}{2j\sigma^2} + \tfrac12\log(\tfrac12\pi n) + \frac{x^2}{2n} \right) \, dx \\
&\leq D(q_{2,+} | \varphi_{\mu,\sigma^2}) + \myo(\log n + j + 1) \\
&= \int_0^\infty q_{2,+} \log \left( \frac{q_{2,+}}{\varphi_+} \frac{\varphi_+}{\varphi_{\mu,\sigma^2}} \right) \, dx + \myo(\log n + j + 1) \\
&= D(q_{2,+} | \varphi_+) + \int_0^\infty q_{2,+} \log \left( \frac{\varphi_+}{\varphi_{\mu,\sigma^2}} \right) \, dx + \myo(\log n + j + 1) \\
% &= D(q_{2,+} | \varphi_+) + \int_0^\infty q_{2}(x) \left( -\tfrac12\log(\tfrac12\pi) - \tfrac12x^2 + \tfrac12\log(2\pi \sigma^2) + \frac{(x-\mu)^2}{2\sigma^2} \right) \, dx + \myo(\log n + j + 1) \\
&= \myo(\log n + j + 1) \,,
\end{align*}
the implicit constants depending only on $q_{2,+}$.
Here the last step follows from \eqref{eq:ECLT-2},
see the~remark below Lemma \ref{lemma:convexity}.

Combining the preceding estimates, it follows that
\begin{multline*}
     D\left(\left.\sum_{k=1}^{n} \varrho^k \, q_{k,2} \ast G_{n-k}^{-} \right| \varphi_{n,+}\right) \\
\leq \sum_{k=1}^{n-1} \varrho^k \, \overline{F}_{n-k}(0) \, \myo(\log n + k + 1) + \myo(\log n / \sqrt{n})
   = \myo(\log n/\sqrt{n}) \,,
\end{multline*}
and the proof of \eqref{eq:entropy-r} is complete.

\pagebreak[2]
\medskip
\myparagraph{Analysis of $\pmb{(\overline{q}_{n}^*)^{\pm}(x)}$}
To complete the proof of part (c), we will show
that the relative entropy of the main terms 
$\overline{p}_n^*$ and $(\overline{q}_n^*)^+$ in \eqref{eq:key-2}
is ``stable'' \wrt the~addition of the error terms 
$\overline{r}_{n,1}^*$, $\overline{r}_{n,2}^*$ and $(\overline{q}_n^*)^-$.
To begin with, it follows from~\eqref{eq:key-2} that
$$
(\overline{q}_n^*)^{+} \leq \overline{p}^*_n + \overline{r}^*_{n,2}
\quad\text{and}\quad
(\overline{q}_n^*)^{-} \leq \overline{r}^*_{n,1}
$$
and therefore,
since $\|\overline{p}_n^*\|_1 = 1 - \overline{F}_n(0) = 1 + \myo(1/\sqrt{n})$ and $\|\overline{r}_{n,j}^*\|_1 = \myo(1/\sqrt{n})$ ($j=1,2$),
\begin{align}
\label{eq:tvq}
\|(\overline{q}_{n}^*)^{+}\|_1 = 1 + \myo(1/\sqrt{n}) 
\quad\text{and}\quad
\|(\overline{q}_{n}^*)^{-}\|_1 = \myo(1/\sqrt{n}) \,.
\end{align}

Next we will show that
\begin{align}
\label{eq:entropy-qminus}
D((\overline{q}_n^*)^- \,|\, \varphi_+) = D(\overline{q}_n^- \,|\, \varphi_{n,+}) = o(1) \,.
\end{align}
Since $q_1$ is bounded by construction,
$(1-\varrho^k) q_{k,1}$ is bounded uniformly in $k \geq 1$,
and we obtain
$$
\|\overline{q}_{n}\|_\infty = \|\sum_{k=1}^{n} (1-\varrho^k) q_{k,1} \ast G_{n-k}\|_\infty = \myo \left( \sum_{k=1}^{n} \frac{1}{\sqrt{n-k+1}} \right) = \myo(\sqrt{n}) \,.
$$
Since $\|\overline{q}_n^-\|_1 = \myo(1/\sqrt{n})$,
it~follows that
\begin{multline*}
  D(\overline{q}_n^- | \varphi_{n,+}) 
= \int_0^\infty \overline{q}_n^- \log(\overline{q}_n^- / \varphi_{n,+}) \, dx
\leq \int_0^\infty \overline{q}_n^- \log(C_1 \sqrt{n} / \varphi_{n,+}) \, dx \\
= \myo\Big(\frac{\log n}{\sqrt{n}}\Big) + \myo\Big(\int_0^\infty \tfrac1n x^2 \, \overline{q}_n^-(x) \, dx\Big) \,.
\end{multline*}
Now, using \eqref{eq:key-2} and \eqref{eq:smr}, we have
$$
     \int_0^\infty \tfrac1n x^2 \overline{q}_n^-(x) \, dx 
\leq \int_0^\infty \tfrac1n x^2 \overline{r}_{n,1}(x) \, dx 
= \int_0^\infty y^2 \overline{r}_{n,1}^*(y) \, dy
= \myo(n^{-3/2}) \,.
$$
This completes the proof of \eqref{eq:entropy-qminus}.

Using \eqref{eq:key-2}, \eqref{eq:entropy-r}, \eqref{eq:entropy-qminus} 
as well as Lemma \ref{lemma:perturb}, we now obtain
\begin{align*}
   D(\overline{p}_n^* | \varphi_+) 
&= D(\overline{p}_n^* + (\overline{q}_n^*)^- + \overline{r}_{n,2}^* | \varphi_+) + o(1) \\
&= D((\overline{q}_n^*)^+ + \overline{r}_{n,1}^* | \varphi_+) + o(1) \\ 
&= D((\overline{q}_n^*)^+ | \varphi_+) + o(1)
\end{align*}
as $n \to \infty$, and Lemma \ref{lemma:entropy-q} is proved.
\end{proof}

\pagebreak[2]
\medskip
\section{Proof of Sufficiency in Theorem \ref{thm:ECLT}}

This section contains the main part 
of the proof of sufficiency in Theorem \ref{thm:ECLT}.
It relies on two auxiliary results 
which do not depend on condition \eqref{eq:ECLT-2}
and whose proof is postponed to the following sections.

\begin{proposition}
\label{prop:intlimit}
For any $\varepsilon > 0$, there exists a constant $C > 0$
such that
$$
\int_{C}^{\infty} x^2 \, \overline{p}_n^*(x) \, dx \leq \varepsilon
$$
for all sufficiently large $n \in \mynat$.
\end{proposition}

\begin{proposition}
\label{prop:local}
Under the assumptions of Theorem \ref{thm:ECLT}, 
there exist signed \linebreak densities $r_n(x)$ such that 
$\|r_n\|_1 = \myo(1/\sqrt{n})$, $\|r_n\|_\infty = \myo(1)$
and the following holds:
\begin{enumerate}[(a)]
\item
Uniformly in $x \in (0,\infty)$,
$$
\overline{q}_n^*(x) = \varphi_+(x) + r_n(x) + o(1/x) \qquad \text{as $n \to \infty$} \,.
$$
\item
Uniformly in $x \in (0,e^{-1})$,
$$
\overline{q}_n^*(x) = \varphi_+(x) + r_n(x) + \myo\left(\log n \wedge \tfrac{1}{\sqrt{n} x}\right) + \myo\left(\log x^{-1}\right) \qquad \text{as $n \to \infty$} \,.
$$
\end{enumerate}
\end{proposition}

Here the norms $\|\,\cdot\,\|_1$ and $\|\,\cdot\,\|_\infty$
are defined as in the proof of Lemma \ref{lemma:entropy-q}.
More\-over, by the statement that the $\myo$-bounds and $o$-bounds
hold uniformly in~$x$, we mean that for sufficiently large $n \in \mynat$, 
the error term is bounded 
by $\varepsilon_n/x$ in~part~(a), 
where $(\varepsilon_n)_{n \in \mynat}$ is a sequence of positive real numbers
not depending on $x \in (0,\infty)$ such that $\lim_{n \to \infty} \varepsilon_n = 0$,
and by $C_1 \left( \log n \wedge \tfrac{1}{\sqrt{n} x} \right) + C_2 \left( \log x^{-1} \right)$ in~part~(b),
where $C_1$ and $C_2$ are positive constants not depending on~$x \in (0,e^{-1})$.
Similar conventions apply to the error terms
in the proof of Proposition \ref{prop:local}.

Note that Proposition \ref{prop:local} may be regarded as a local version 
of the Erd{\H{o}}s-Kac theorem \eqref{thm:erdos-kac}.
Moreover, part (b) is a refinement of part (a)
which yields a~better estimate for the error term
for $x \approx 0$.
Although this estimate is still unbounded,
it is square-integrable near the origin.
This is the crucial point for our purposes.

It should be mentioned that the proof of Proposition~\ref{prop:local} 
closely follows that in Aleshkyavichene \cite{Alesh:77},
which is based on earlier work by Nagaev \cite{Nagaev:69,Nagaev:70,Nagaev:70a}.
Indeed, in the special case where the $X_j$ 
have a bounded density $p(x)$, we~could~take 
$$
\overline{q}_n^*(x) := \overline{p}_n^*(x)
\quad\text{and}\quad
r_n(x) := \overline{F}_{n-1}(0) \, \sqrt{n} \, p(\sqrt{n} \, x) \quad (x > 0) \,,
$$
and part (a) specializes to the following result
from the literature:

\begin{theorem}[Aleshkyavichene \cite{Alesh:77}]
\label{thm:Alesh}
If $X_1,X_2,\dots$ have a bounded density $p(x)$, we have
$\overline{p}_n^*(x) = \varphi_+(x) + \overline{F}_{n-1}(0) \, \sqrt{n} p(\sqrt{n} x) + o(1/x)$,
uniformly in $x \in (0,\infty)$.
\end{theorem}

\begin{remark*}
In \cite{Alesh:77} Theorem~\ref{thm:Alesh} is stated somewhat differently
(for any $x_0 > 0$, the~last~term is of order $o(1)$ uniformly in $x > x_0$),
but a careful analysis of the~proof
shows that after some minor modifications (similar to~those 
in the~proof of part~(a) of Proposition \ref{prop:local} below),
it also yields the result stated above.
\end{remark*}

In the general case, the definition of the signed densities $r_n(x)$ 
is more complicated, see Equation \eqref{eq:rndef} below.

\pagebreak[2]

\begin{proof}[Proof of Sufficiency in Theorem \ref{thm:ECLT}]
Suppose that \eqref{eq:ECLT-2} holds.
Recall that $\overline{p}_n^*(x)$ 
is the density of $\overline{S}_n/\sqrt{n}$
(\ie with the proper~rescaling),
and $\varphi_+(x) = \sqrt{2/\pi} \, e^{-x^2/2}$ ($x > 0$).
Using \eqref{eq:osn}, it is easy to~see that
\begin{align}
\label{eq:simplify}
D_+(\overline{S}_n/\sqrt{n} \,|\, |Z|) \to 0
\quad\text{if and only if}\quad
D(\overline{p}_n^* \,|\, \varphi_+) \to 0 \,.
\end{align}
Indeed, since $\overline{S}_n/\sqrt{n}$ conditioned to be positive
has the density $\overline{p}_n^*(x)/(1-\overline{F}_n(0))$ ($x > 0$),
it follows from our definitions and Lemma \ref{lemma:scalar} that
$$
  D(\overline{p}_n^* \,|\, \varphi_+)
= (1-\overline{F}_n(0)) \, D_+(\overline{S}_n/\sqrt{n} \,|\, |Z|) + L(1-\overline{F}_n(0)) \,,
$$
so that \eqref{eq:simplify} follows from \eqref{eq:osn}.

\pagebreak[2]

Since $D_+(\overline{S}_n/\sqrt{n} \,|\, |Z|) \geq 0$,
it also follows from the preceding argument that
$$
\liminf_{n \to \infty} D(\overline{p}_n^* \,|\, \varphi_+) \geq 0 \,.
$$
Thus, it remains to show that
$$
\limsup_{n \to \infty} D(\overline{p}_n^* \,|\, \varphi_+) \leq 0 \,.
$$

Recall that $\overline{q}_n^*(x) := \sqrt{n} \, \overline{q}_n(\sqrt{n} \, x)$,
where $\overline{q}_n$ is defined in \eqref{eq:qrdef}.
By~Lemma~\ref{lemma:entropy-q}\,(c),
it is sufficient to~show that 
$$
\limsup_{n \to \infty} D((\overline{q}_n^*)^+ \,|\, \varphi_+) \leq 0 \,.
$$
Fix $\varepsilon_0 > 0$, and let $C$ and $c$ 
be positive real numbers with $0 < c < 1 < C < \infty$.
(The precise choices will be specified below.)
Then
$$
D((\overline{q}_n^*)^+ \,|\, \varphi_+) = \int_0^\infty L\left(\frac{(\overline{q}_n^*)^+(x)}{\varphi_+(x)}\right) \varphi_+(x) \, dx = E_1 + E_2 + E_3 \,,
$$
where $E_1,E_2,E_3$ denote the integrals over the intervals
$(0,c)$, $(c,C)$, $(C,\infty)$, respectively.
(Note that $E_1,E_2,E_3$ implicitly depend on $n$.)
To complete the proof, we will show
that if $C \in (1,\infty)$ is sufficiently large
and $c \in (0,1)$ is sufficiently small, \linebreak[2]
then, for each $j \in \{ 1,2,3 \}$, 
$E_j \leq \varepsilon_0$ for all sufficient\-ly~large $n \in \mynat$.

\pagebreak[2]
\medskip

\textbf{Estimating $\pmb{E_3}$.}
By Proposition \ref{prop:local}\,(a), there exists a constant $M > 1$ 
\linebreak (not depending on $n$) such~that for $n \geq n_0$ and 
$x \geq 1$, $|\overline{q}_n^*(x)| \leq M$. It follows that
$$
     E_3
\leq \int_{C}^{\infty} |\overline{q}_n^*(x)| (\log M + \tfrac12 \log \tfrac\pi2 + \tfrac12 x^2) \, dx
\leq C_1 \int_{C}^\infty x^2 |\overline{q}_n^*(x)| \, dx \,,
$$
where $C_1$ is a constant depending only on $M$.
By Proposition \ref{prop:intlimit}, there exists a~constant $C > 1$ 
such that
\begin{align*}
% \label{eq:p-int}
\int_{C}^\infty x^2 |\overline{p}_n^*(x)| \, dx < \varepsilon_0/C_1
\end{align*}
for all sufficiently large $n \in \mynat$.
By Lemma \ref{lemma:entropy-q}\,(b), this implies
\begin{align*}
% \label{eq:q-int}
\int_{C}^\infty x^2 |\overline{q}_n^*(x)| \, dx < \varepsilon_0/C_1
\end{align*}
for all sufficiently large $n \in \mynat$.
Thus, for $C$ sufficiently large, we have $E_3 \leq \varepsilon_0$ 
for~all sufficiently large $n \in \mynat$.

\pagebreak[2]
\medskip

\textbf{Estimating $\pmb{E_1}$.}
Suppose that $c \in (0,e^{-1})$.
Setting 
$$
v_n(x) := \frac{(\overline{q}_n^*)^+(x) - \varphi_+(x)}{\varphi_+(x)} \qquad (x > 0)
$$
and using that $L(y) \leq 0$ for $y \in [0,1]$
and $L(1+y) \leq y + \tfrac12 y^2$ for $y \in (0,\infty)$, 
we~get
$$
     E_1
   = \int_0^c L \left( 1+v_n(x) \right) \varphi_+(x) \, dx
\leq \int_0^c \, \left(|v_n(x)| + \tfrac12|v_n(x)|^2\right) \varphi_+(x) \, dx \,.
$$
Using Proposition~\ref{prop:local}\,(b), it follows that
\begin{align*}
      E_1
&\leq \int_{0}^{c} |\overline{q}_n^*(x) - \varphi_+(x)| + \tfrac12|\overline{q}_n^*(x) - \varphi_+(x)|^2 / \varphi_+(x) \, dx \\
&\leq C_2 \bigg( \int_{0}^{c} |r_n(x)| \, dx + \int_{0}^{c} (\log n \wedge \tfrac{1}{\sqrt{n} x}) \, dx + \int_{0}^{c} (\log x^{-1}) \, dx \bigg) \\
&\qquad \,+\, C_3 \bigg( \int_{0}^{c} |r_n(x)|^2 \, dx + \int_{0}^{c} (\log n \wedge \tfrac{1}{\sqrt{n} x})^2 \, dx + \int_{0}^{c} (\log x^{-1})^2 \, dx \bigg) \,.
% &\leq C_2 \bigg( \|r_n\|_1 + o(1) (\log C - \log c) \bigg) + C_3 \exp(C^2/2) \bigg( \|r_n\|_1 \cdot \|r_n\|_\infty + o(1) (c^{-1} - C^{-1}) \bigg) 
% = o(1) \,.
\end{align*}
By Cauchy-Schwarz inequality, it remains to control
the integrals in the last line.
Now, for any fixed $c \in (0,e^{-1})$, we have
$$
\int_0^c |r_n(x)|^2 \, dx \leq \|r_{n}\|_1 \|r_{n}\|_\infty = o(1) \,,
$$
$$
\int_0^c (\log n \wedge \tfrac{1}{\sqrt{n} x})^2 \, dx = \frac{\log n}{\sqrt{n}} + \frac{1}{n} (-c^{-1} + \sqrt{n} \log n) = o(1) \,,
$$
$$
\int_0^c (\log x^{-1})^2 \, dx = \int_{\log(1/c)}^{\infty} y^2 e^{-y} \, dy < \infty \,.
$$
Thus, for $c$ sufficiently small, we have $E_1 \leq \varepsilon_0$ 
for all sufficiently large $n \in \mynat$.

\pagebreak[2]
\medskip

\textbf{Estimating $\pmb{E_2}$.}
Let $C \in (1,\infty)$ and $c \in (0,1)$ be the constants fixed above.
The same argument as for $E_1^+$ yields
$$
     E_2
   = \int_c^C L \left( 1+v_n(x) \right) \varphi_+(x) \, dx
\leq \int_c^C \, \left(|v_n(x)| + \tfrac12|v_n(x)|^2\right) \varphi_+(x) \, dx \,.
$$
Using Proposition~\ref{prop:local}\,(a), it follows that
\begin{align*}
      E_2
&\leq \int_{c}^{C} |\overline{q}_n^*(x) - \varphi_+(x)| + \tfrac12|\overline{q}_n^*(x) - \varphi_+(x)|^2 / \varphi_+(x) \, dx \\
&\leq C_4 \bigg( \int_{c}^{C} |r_n(x)| \, dx + o(1) \int_{c}^{C} x^{-1} \, dx \bigg) \\&\qquad \,+\, C_5 \exp(C^2/2) \bigg( \int_{c}^{C} |r_n(x)|^2 \, dx + o(1) \int_{c}^{C} x^{-2} \, dx \bigg) \\ 
&\leq C_4 \bigg( \|r_n\|_1 + o(1) (\log C - \log c) \bigg) \\&\qquad \,+\, C_5 \exp(C^2/2) \bigg( \|r_n\|_1 \cdot \|r_n\|_\infty + o(1) (c^{-1} - C^{-1}) \bigg) \,.
\end{align*}
Thus, $E_2^+ = o(1)$ as $n \to \infty$.

\medskip

This completes the proof of sufficiency in Theorem \ref{thm:ECLT}.
% (up to the proof of Pro\-positions \ref{prop:intlimit} and \ref{prop:local}).
\end{proof}

\pagebreak[2]
\medskip

\section{Some Auxiliary Results}
\label{sec:aux}

Let us collect some results from the literature
which will be needed for the proofs of Propositions
\ref{prop:intlimit} and \ref{prop:local}.

Let $\overline{a}_k := \int_{-\infty}^{0} x \, d\overline{F}_k(x)$
and $\overline{b}_k := \int_{-\infty}^{0} x^2 \, d\overline{F}_k(x)$,
$k \geq 1$. It is known that under our standing moment assumptions
the functions $\overline\varphi_k(t)$ introduced in~\eqref{eq:varphidef}
satisfy the following estimates:
\begin{align}
\label{eq:phi-1}
|\overline\varphi_k(t)| &\leq 2\overline{F}_k(0) \,,
\\
\label{eq:phi-2}
|\overline\varphi_k(t)| &\leq |\overline{a}_k||t| \,,
\\
\label{eq:phi-3}
|\overline\varphi_k'(t)| &\leq |\overline{a}_k| \,,
\\
\label{eq:phi-4}
|\overline\varphi_k(t) - (-it\overline{a}_k)| &\leq \tfrac12 |\overline{b}_k||t|^2 \,,
\\
\label{eq:phi-5}
|\overline\varphi_k'(t) - (-i\overline{a}_k)| &\leq |\overline{b}_k||t| \,,
\\
\label{eq:phi-6}
|\overline\varphi_k''(t)| &\leq |\overline{b}_k| \,,
\end{align}
(see \eg \cite[Equations (26) and (46)]{Alesh:77}),
where 
\begin{align}
\label{eq:fzero}
\overline{F}_k(0) = \myo(k^{-1/2})
\end{align}
(see \eg \cite[Equation (39)]{Alesh:77}),
\begin{align}
\label{eq:ab}
\overline{a}_k = - (2\pi k)^{-1/2} + o(k^{-1/2})
\quad\text{and}\quad
\overline{b}_k = o(1)
\end{align}
(see \eg \cite[Equation (1)]{Alesh:77}).
Let us note that the implicit constants 
may depend on the distribution of $X_1$.

\pagebreak[2]

Furthermore, we need the following classical approximations
for characteristic functions of sums of \iid random variables
and their derivatives:

Given \iid random variables $X_1,X_2,X_3,\hdots$ with 
mean~$0$, variance~$1$, density~$p$ and characteristic function $f$,
there exist positive real numbers $\gamma,\delta_1,\delta_2,\delta_3, \linebreak[1] \dots$ 
(depending on the distribution of $X_1$)
with $\lim_{n \to \infty} \delta_n = 0$ 
such that for $n \in \mynat$, $|t| \leq \gamma n^{1/2}$ and $j = 0,1,2$,
\begin{align*}
\left|\tfrac{d^j}{dt^j} (f^n(t/\sqrt{n}) - e^{-t^2/2})\right| &\leq \delta_n \, e^{-t^2/4} \,.
\end{align*}
See \eg \cite[Theorem 9.12]{BR}. 
Replacing $n$ with $k$ and $t$ with $t \sqrt{k/n}$ in this estimate, 
we~obtain, for $1 \leq k \leq n$, $|t| \leq \gamma n^{1/2}$ and $j = 0,1,2$,
\begin{align}
\label{eq:fk-all}
\left|\tfrac{d^j}{dt^j}(f^k(t/\sqrt{n}) - e^{-kt^2/2n})\right| &\leq \delta_k \, (k/n)^{j/2} e^{-kt^2/4n} \,.
\end{align}
Furthermore, let $\eta \in (0,1)$ be a constant 
such that 
\begin{align}
\label{eq:RL}
|t| \geq \gamma \quad \Rightarrow \quad |f(t)| \leq \eta \,.
\end{align}
Such a constant $\eta$ exists because $X_1$ has a density,
which implies that $|f(t)| < 1$ for~all $t \ne 0$
as well as $\lim_{|t| \to \infty} |f(t)| = 0$
(by the Riemann-Lebesgue lemma).

\pagebreak[4]  % DANGER

Besides that, we will repeatedly use the fact that
for any $\alpha > 0$ and $n \geq k \geq 1$,
\begin{equation}
\label{eq:GF}
\sup_{t \in \myreal} \, (kt^2/n)^{\alpha/2} \, e^{-kt^2/4n}
= \myo_\alpha(1)
\end{equation}
and
\begin{equation}
\label{eq:GI}
\int_{-\infty}^{+\infty} (kt^2/n)^{\alpha/2} \, e^{-kt^2/4n} \, dt
= \myo_\alpha(\sqrt{\tfrac{n}{k}}) \,,
\end{equation}
with implicit constants depending only on $\alpha$.

In addition to that, we will use the following (well-known) Gaussian tail bounds:
For any $\alpha > 0$ and $t > 0$ we have
\begin{align}
\label{eq:gtb-0}
     \int_t^\infty e^{-\alpha x^2/2} \, dx
&\leq \sqrt{\tfrac{\pi}{2\alpha}} \wedge \left( \tfrac{1}{\alpha t} e^{-\alpha t^2/2} \right) \,,
\\
\label{eq:gtb-1}
   \int_t^\infty \sqrt{\alpha} x e^{-\alpha x^2/2} \, dx
&= \tfrac{1}{\sqrt{\alpha}} e^{-\alpha t^2/2} \,,
\\
\label{eq:gtb-2}
     \int_t^\infty \alpha x^2 e^{-\alpha x^2/2} \, dx
&\leq \sqrt{\tfrac{\pi}{2\alpha}} \wedge \left( \tfrac{1}{\alpha t} (\alpha t^2 + 1) e^{-\alpha t^2/2} \right) \,.
\end{align}
% These estimates will be used without further notice.

Moreover, we will repeatedly use the fact that
\begin{align}
\label{eq:betasum-1}
  \sum_{k=1}^{n-1} \frac{1}{\sqrt{k(n\!-\!k)}}
= \myo\left( \frac{1}{\sqrt{n}} \sum_{1 \leq k \leq n/2} \frac{1}{\sqrt{k}} \right) + \myo\left( \frac{1}{\sqrt{n}} \sum_{n/2 \leq k \leq n-1} \frac{1}{\sqrt{n\!-\!k}} \right)
= \myo(1) \,.
\end{align}
A similar decomposition shows that if $(t_n)_{n \in \mynat}$ 
is a sequence of real numbers with $\lim_{n \to \infty} t_n = 0$, 
we have
\begin{align}
\label{eq:betasum-2}
\sum_{k=1}^{n-1} \frac{t_k}{\sqrt{k(n-k)}} = o(1) \,.
\end{align}

Finally, we will need the observation that the Fourier transform 
of the density $\varphi_+(x) := \sqrt{2/\pi} e^{-x^2/2}$ 
($x > 0$) satisfies
\begin{align}
\label{eq:transform}
   \hat\varphi_+(t) 
% &= e^{-t^2/2} \left[ 1 + \sqrt{2/\pi} \, i \int_0^t e^{u^2/2} \, du \right] \nonumber \\
&= e^{-t^2/2} + \frac{it}{\sqrt{2\pi n}} \int_0^n e^{-ut^2/2n} \, \frac{du}{\sqrt{n-u}}
\end{align}
for all $n \in \mynat$
(see \cite[page 452]{Alesh:77}).
It follows from this that for any $x > 0$,
\begin{align}
\varphi_+(x) = \frac{1}{2\pi} \lim_{R \to \infty} \int_{-R}^{+R} e^{-itx} \left[ e^{-t^2/2} + \frac{it}{\sqrt{2\pi n}} \int_0^n e^{-ut^2/2n} \, \frac{du}{\sqrt{n-u}} \right] \, dt 
\label{eq:invtransform}
\end{align}
(see \cite[page 452]{Alesh:77}).

\pagebreak[2]
\medskip
\section{Proof of Proposition \ref{prop:intlimit}}

Proposition \ref{prop:intlimit} will be deduced from the following result:

\begin{proposition}
\label{prop:charfunc}
For $k=0,1,2,$ we have
$$
\frac{d^k}{dt^k} \Big[ \ee(e^{it\overline{S}_n/\sqrt{n}}) - \hat\varphi_+(t) \Big] = o(1)
$$
as $n \to \infty$, uniformly in $|t| \leq \gamma n^{1/2}$.
\end{proposition}

\begin{remarks} \

(a) The Erd{\H{o}}s-Kac theorem is equivalent to the statement that 
$\ee(e^{it\overline{S}_n/\sqrt{n}}) \to \hat\varphi_+(t)$
for any fixed $t \in \myreal$.
Thus, this theorem follows from Proposition \ref{prop:charfunc}.
Let us emphasize that we do not need the existence of densities in this section.

(b) For our ``application'' 
(namely the proof of Proposition \ref{prop:intlimit}),
the result for the second derivative is relevant.
Indeed, for this application, it~would~be be sufficient to prove 
Proposition \ref{prop:charfunc} for $t = \myo(1)$.
\end{remarks}

\begin{proof}[Proof of Proposition \ref{prop:charfunc}]
Similarly as in \cite{Alesh:77,Naud:77},
using \eqref{eq:nagaev}~and~\eqref{eq:transform},
we have the following decomposition:
\begin{align*}
   \ee(e^{it\overline{S}_n/\sqrt{n}}) - \hat\varphi_+(t)
&= \bigg[ f^n(t/\sqrt{n}) - e^{-t^2/2} \bigg] \\
&+ \bigg[ \frac{it}{\sqrt{2\pi n}} \bigg( \sum_{k=3}^{n-1} e^{-kt^2/2n} \frac{1}{\sqrt{n-k}} - \int_0^n e^{-ut^2/2n} \frac{du}{\sqrt{n-u}} \bigg) \bigg] \\
&+ \bigg[ \sum_{k=3}^{n-1} \Big( f^k(t/\sqrt{n}) - e^{-kt^2/2n} \Big) \overline\varphi_{n-k}(t/\sqrt{n}) \bigg] \\
&+ \bigg[ \sum_{k=3}^{n-1} e^{-kt^2/2n} \Big( \overline\varphi_{n-k}(t/\sqrt{n}) - (-\overline{a}_{n-k}) \, it/\sqrt{n} \Big) \bigg] \\
&+ \bigg[ \sum_{k=3}^{n-1} e^{-kt^2/2n} \Big( (-\overline{a}_{n-k}) - \frac{1}{\sqrt{2\pi(n\!-\!k)}} \Big) \, it/\sqrt{n} \bigg] \\
&+ \bigg[ f^2(t/\sqrt{n}) \overline\varphi_{n-2}(t/\sqrt{n}) + f(t/\sqrt{n}) \overline\varphi_{n-1}(t/\sqrt{n}) \bigg] \,.
\end{align*}
Denote the expressions in the square brackets by $D_1(t),\hdots,D_6(t)$.
(Note that all these expressions implicitly depend on $n$.)
We will show that for $j=1,\hdots,6$, uniformly in $|t| \leq \gamma n^{1/2}$, \ 
$D_j(t),D_j'(t),D_j''(t) \to 0$ \ as $n \to \infty$.

\smallskip

\emph{Convention:}
We always assume that $n \geq 4$ and $|t| \leq \gamma n^{1/2}$.
$\myo$- and $o$-\linebreak[2]bounds hold uniformly in this region
(unless otherwise mentioned), and they may depend on the constants 
$\gamma,\delta_1,\delta_2,\delta_3,\dots$ 
introduced in Section \ref{sec:aux}.

\medskip
\myparagraph{On the Difference $\pmb{D_1}$}
For the difference $D_1(t)$ and its first two derivatives, 
the~claim is immediate from \eqref{eq:fk-all} (with $k = n$).

\medskip
\myparagraph{On the Difference $\pmb{D_2}$}
For fixed $n \in \mynat$, $t \in \myreal$ and $\beta \in \{ 0,1,2,\hdots \}$, put
$$
h_{\beta}(u) := (u/n)^\beta \, e^{-ut^2/2n} \, \tfrac{1}{\sqrt{n-u}} \qquad (0<u<n) \,.
$$
Then, for $1 \leq v \leq w \leq n-1$, we have
\begin{multline*}
\left| h_{\beta}(w) - h_{\beta}(v) \right|
=
\left| \int_v^w h'_{\beta}(u) \, du \right|
=
\left| \int_v^w \left( \tfrac{\beta}{u} - \tfrac{t^2}{2n} + \tfrac{1}{2(n-u)} \right)  h_\beta(u) \, du \right| \\[+3pt]
\leq
(w-v) \left( \tfrac{\beta}{v} + \tfrac{t^2}{2n} + \tfrac{1}{2(n-w)} \right) (w/n)^\beta \, e^{-vt^2/2n} \, \tfrac{1}{\sqrt{n-w}} \,.
\end{multline*}
Hence, for the difference $D_2(t)$, we get (using the above estimate with $\beta = 0$)
\begin{align*}
&\eskip \left| \frac{it}{\sqrt{2\pi n}} \bigg( \sum_{k=3}^{n-1} e^{-kt^2/2n} \frac{1}{\sqrt{n-k}} - \int_0^n e^{-ut^2/2n} \frac{du}{\sqrt{n-u}} \bigg) \right| \\ 
&\leq \frac{|t|}{\sqrt{2\pi n}} \sum_{k=3}^{n-2} \left| \int_{k}^{k+1} e^{-kt^2/2n} \tfrac{1}{\sqrt{n-k}} - e^{-ut^2/2n} \tfrac{1}{\sqrt{n-u}} \, du \right| + \myo(n^{-1/2}) \\
&\leq \frac{|t|}{\sqrt{2\pi n}} \sum_{k=3}^{n-2} \left( \tfrac{1}{(n-k-1)^{3/2}} e^{-kt^2/2n} + \tfrac{1}{(n-k-1)^{1/2}} \tfrac{t^2}{2n} e^{-kt^2/2n} \right) + \myo(n^{-1/2}) \\
&= \myo \bigg( \sum_{k=3}^{n-2} \left( \tfrac{1}{k^{1/2} (n-k-1)^{3/2}} + \tfrac{1}{k^{3/2} (n-k-1)^{1/2}} \right) \bigg) + \myo(n^{-1/2})
= \myo(n^{-1/2}) \,.
\end{align*}
Here we have used the fact that $(k/n)^{1/2} \, |t| \, e^{-kt^2/2n}$
and $(k/n)^{3/2} \, |t|^3 \, e^{-kt^2/2n}$ are uniformly bounded. 
In particular, this fact is also used in the first step to absorb 
the summand for $k=n-1$ and the integral over $u \in [n-1,n]$
into the $\myo(n^{-1/2})$-term.

Furthermore, similar estimates hold for the first two derivatives of $D_2(t)$.
Indeed, these derivatives are finite linear combinations of expressions
of the form
$$
\frac{it^\alpha}{\sqrt{2\pi n}} \bigg( \sum_{k=3}^{n-1} (k/n)^\beta e^{-kt^2/2n} \frac{1}{\sqrt{n-k}} - \int_0^n (u/n)^\beta e^{-ut^2/2n} \frac{du}{\sqrt{n-u}} \bigg) \\ 
$$
(with $\alpha,\beta \in \{ 0,1,2,3,\hdots \}$ and $\alpha \leq \beta + 1$),
and, by similar arguments as above,
\begin{align*}
&\eskip \left| \frac{it^\alpha}{\sqrt{2\pi n}} \bigg( \sum_{k=3}^{n-1} (k/n)^\beta e^{-kt^2/2n} \frac{1}{\sqrt{n-k}} - \int_0^n (u/n)^\beta e^{-ut^2/2n} \frac{du}{\sqrt{n-u}} \bigg) \right| \\ 
&\leq \frac{|t|^\alpha}{\sqrt{2\pi n}} \sum_{k=3}^{n-2} \left| \int_{k}^{k+1} (k/n)^\beta e^{-kt^2/2n} \tfrac{1}{\sqrt{n-k}} - (u/n)^\beta e^{-ut^2/2n} \tfrac{1}{\sqrt{n-u}} \, du \right| + \myo_\beta(n^{-1/2}) \\
&\leq \frac{|t|^\alpha}{\sqrt{2\pi n}} \sum_{k=3}^{n-2} \left( \tfrac{((k+1)/n)^\beta}{(n-k-1)^{3/2}} e^{-kt^2/2n} + \tfrac{((k+1)/n)^\beta}{(n-k-1)^{1/2}} \tfrac{t^2}{2n} e^{-kt^2/2n} + \tfrac{((k+1)/n)^\beta}{(n-k-1)^{1/2}} \tfrac{\beta}{k} e^{-kt^2/2n} \right) \\&\mskip490mu + \myo_\beta(n^{-1/2}) \\
&= \myo_\beta \bigg( \sum_{k=3}^{n-2} \left( \tfrac{1}{k^{1/2} (n-k-1)^{3/2}} + \tfrac{1}{k^{3/2} (n-k-1)^{1/2}} + \tfrac{1}{k^{3/2} (n-k-1)^{1/2}} \right) \bigg) + \myo_\beta(n^{-1/2}) \\
&= \myo_\beta(n^{-1/2}) \,.
\end{align*}
% Here we have repeatedly used the fact that for fixed $\alpha,\beta$ 
% with $\alpha \leq \beta + 1$, \linebreak
% $(k/n)^{\beta+(1/2)} \, t^\alpha \, e^{-kt^2/n}$ is uniformly bounded.
% In particular, this is also used in the first step to absorb 
% the summand for $k=n-1$ and the integral over $u \in [n-1,n]$
% into the $\myo(n^{-1/2})$-term.

\medskip
\myparagraph{On the Difference $\pmb{D_3}$}
For the difference $D_3(t)$, the claim follows
from \eqref{eq:fk-all} (with $k < n$),
\eqref{eq:phi-2}, \eqref{eq:ab} and \eqref{eq:betasum-2},
since
\begin{multline*}
\eskip \sum_{k=3}^{n-1} \Big( f^k(t/\sqrt{n}) - e^{-kt^2/2n} \Big) \overline\varphi_{n-k}(t/\sqrt{n}) \\
= \myo\left(\sum_{k=3}^{n-1} \frac{\delta_k \, (k/n)^{1/2} |t| \, e^{-kt^2/4n}}{\sqrt{k(n-k)}}\right)
= \myo\left(\sum_{k=3}^{n-1} \frac{\delta_k}{\sqrt{k(n-k)}}\right)
= o(1) \,.
\end{multline*}
Similar estimates hold for the first two derivatives.
Indeed, using \eqref{eq:fk-all} (with $k < n$),
\eqref{eq:phi-2} -- \eqref{eq:phi-3}, \eqref{eq:phi-6},
\eqref{eq:ab} and \eqref{eq:betasum-2}, we get
\begin{align*}
&\eskip 
   \sum_{k=3}^{n-1} \frac{d}{dt} \bigg[ \Big( f^k(t/\sqrt{n}) - e^{-kt^2/2n} \Big) \overline\varphi_{n-k}(t/\sqrt{n}) \bigg] \\
&= \sum_{k=3}^{n-1} \bigg[ \frac{d}{dt} \Big( f^k(t/\sqrt{n}) - e^{-kt^2/2n} \Big) \overline\varphi_{n-k}(t/\sqrt{n}) \\&\qquad\qquad\qquad\qquad + \Big( f^k(t/\sqrt{n}) - e^{-kt^2/2n} \Big) \overline\varphi_{n-k}'(t/\sqrt{n}) / \sqrt{n} \bigg] \\[+7pt]
&= \myo\left(\sum_{k=3}^{n-1} \bigg[ \frac{\delta_k \, (k/n)^{1/2} |t| \, e^{-kt^2/4n}}{\sqrt{n(n-k)}} + \frac{\delta_k \, e^{-kt^2/4n}}{\sqrt{n(n-k)}} \bigg]\right) \\
&= \myo\left(\sum_{k=3}^{n-1} \bigg[ \frac{\delta_k}{\sqrt{n(n-k)}} + \frac{\delta_k}{\sqrt{n(n-k)}} \bigg]\right)
= o(1)
\end{align*}
as well as
\begin{align*}
&\eskip 
   \sum_{k=3}^{n-1} \frac{d^2}{dt^2} \bigg[ \Big( f^k(t/\sqrt{n}) - e^{-kt^2/2n} \Big) \overline\varphi_{n-k}(t/\sqrt{n}) \bigg] \\
&= \sum_{k=3}^{n-1} \bigg[ \frac{d^2}{dt^2} \Big( f^k(t/\sqrt{n}) - e^{-kt^2/2n} \Big) \overline\varphi_{n-k}(t/\sqrt{n}) \\
    &\qquad\qquad + 2 \frac{d}{dt} \Big( f^k(t/\sqrt{n}) - e^{-kt^2/2n} \Big) \overline\varphi_{n-k}'(t/\sqrt{n})/\sqrt{n} \\[+7pt]
    &\qquad\qquad\qquad\qquad + \Big( f^k(t/\sqrt{n}) - e^{-kt^2/2n} \Big) \overline\varphi_{n-k}''(t/\sqrt{n}) / n \bigg] \\[+7pt]
&= \myo\Bigg(\sum_{k=3}^{n-1} \bigg[ \frac{\delta_k \, (k/n) \, |t| e^{-kt^2/4n} }{\sqrt{n(n-k)}} 
    + \frac{\delta_k \, (k/n)^{1/2} \, e^{-kt^2/4n}}{\sqrt{n(n-k)}} 
    + \frac{\delta_k \, |\overline{b}_{n-k}| \, e^{-kt^2/4n}}{n} \bigg]\Bigg) \\
&= \myo\Bigg(\sum_{k=3}^{n-1} \bigg[ \frac{\delta_k \sqrt{k/n}}{\sqrt{n(n-k)}} + \frac{\delta_k \sqrt{k/n}}{\sqrt{n(n-k)}} + \frac{\delta_k |\overline{b}_{n-k}|}{n} \bigg]\Bigg)
 = o(1) \,.
\end{align*}

\medskip
\myparagraph{On the Difference $\pmb{D_4}$}
Let $(m_n)_{n \in \mynat}$ be a sequence 
of natural numbers such that $\lim_{n \to \infty} m_n = \infty$ 
and $\lim_{n \to \infty} (m_n / n) \to 0$. Then,
by \eqref{eq:phi-4}, \eqref{eq:phi-2} and \eqref{eq:ab}, we~have
\begin{multline*}
     \sum_{k=3}^{n-1} e^{-kt^2/2n} \Big| \overline\varphi_{n-k}(t/\sqrt{n}) - (-\overline{a}_{n-k}) \, it/\sqrt{n} \Big| \\
\leq \sum_{k=3}^{n-m_n} (t^2/n) e^{-kt^2/2n} |\overline{b}_{n-k}| + \sum_{k=n-m_n}^{n-1} 2 (|t|/\sqrt{n}) e^{-kt^2/2n} |\overline{a}_{n-k}| \\
  =  o\left(\sum_{k=3}^{n-m_n} (t^2/n) e^{-kt^2/2n}\right) + \myo\left(\sum_{k=n-m_n}^{n-1} \frac{1}{\sqrt{k(n-k)}}\right) \,.
\end{multline*}
Since $\sum_{k=3}^{\infty} xe^{-kx}$ is uniformly bounded in $x > 0$, 
it~follows that $D_4(t) = o(1)$. \pagebreak[3]

Similar estimates hold for the first two derivatives.
Indeed, to this end, we have to bound, 
among other terms,
$$
\sum_{k=3}^{n-1} e^{-kt^2/2n} \Big( \overline\varphi_{n-k}'(t/\sqrt{n}) / \sqrt{n} - (-\overline{a}_{n-k}) \, i/\sqrt{n} \Big)
$$
and
$$
\sum_{k=3}^{n-1} e^{-kt^2/2n} \Big( \overline\varphi_{n-k}''(t/\sqrt{n}) / n \Big) \,.
$$
(For the other terms we get similar bounds as for lower-order derivatives
but~with extra factors $kt/n$, which are easily controlled due to 
the exponential factor $e^{-kt^2/2n}$.)
But, using \eqref{eq:phi-5}, \eqref{eq:phi-3}, \eqref{eq:phi-6},
and \eqref{eq:ab}, we get
\begin{multline*}
\sum_{k=3}^{n-1} e^{-kt^2/2n} \Big| \overline\varphi_{n-k}'(t/\sqrt{n}) / \sqrt{n} - (-\overline{a}_{n-k}) \, i/\sqrt{n} \Big| \\
\leq \sum_{k=3}^{n-m_n} (|t|/n) e^{-kt^2/2n} |\overline{b}_{n-k}| + \sum_{k=n-m_n}^{n-1} 2 e^{-kt^2/2n} |\overline{a}_{n-k}| / \sqrt{n} \\
  =  o\bigg(\sum_{k=3}^{n-m_n} \frac{t^2+1}{n} e^{-kt^2/2n}\bigg) + \myo\bigg(\sum_{k=n-m_n}^{n-1} \frac{1}{\sqrt{n(n-k)}}\bigg) 
  =  o(1)
\end{multline*}
as well as
$$
     \sum_{k=3}^{n-1} e^{-kt^2/2n} \Big| \overline\varphi_{n-k}''(t/\sqrt{n}) / n \Big|
\leq \sum_{k=3}^{n-1} (1/n) e^{-kt^2/2n} |\overline{b}_{n-k}| 
\leq \frac{1}{n} \sum_{k=3}^{n-1} |\overline{b}_{n-k}|
  =  o(1) \,.
$$

\medskip
\myparagraph{On the Difference $\pmb{D_5}$}
Similarly as above, let $(m_n)_{n \in \mynat}$ be a sequence \linebreak
of natural numbers such that $\lim_{n \to \infty} m_n = \infty$ 
and $\lim_{n \to \infty} (m_n / n) \to 0$.
Then, using \eqref{eq:ab}, we have
\begin{multline*}
     \sum_{k=3}^{n-1} e^{-kt^2/2n} \Big| (-\overline{a}_{n-k}) - \frac{1}{\sqrt{2\pi(n\!-\!k)}} \Big| \, \Big| it/\sqrt{n} \Big| \\ 
  =  o \left( \sum_{k=3}^{n-m_n} \frac{1}{\sqrt{k(n-k)}} \right) + \myo \left( \sum_{k=n-m_n}^{n-1} \frac{1}{\sqrt{k(n-k)}} \right)
  =  o(1) \,.
\end{multline*}
Again, for the derivatives, we have similar estimates involving 
lower powers of $t$ and\,/\,or additional factors $kt/n$.

\medskip
\myparagraph{On the Difference $\pmb{D_6}$}
For fixed $k$, we have 
$$
|(f^k)(t)| = \myo_k(1) \,,
\qquad
|\tfrac{d}{dt} (f^k)(t)| = \myo_k(1) \,,
\qquad
|\tfrac{d^2}{dt^2} (f^k)(t)| = \myo_k(1)
$$
(as follows from our assumption $\ee X_1^2 < \infty$),
and as $n \to \infty$,
$$
\overline\varphi_{n}(t) = o(1) \,,
\qquad
\overline\varphi_{n}'(t) = o(1) \,,
\qquad
\overline\varphi_{n}''(t) = o(1)
$$
(as follows from \eqref{eq:phi-1}, \eqref{eq:phi-3} and \eqref{eq:phi-6} -- \eqref{eq:ab}).
The claim for the difference $D_6(t)$ and its first two derivatives follows immediately 
from these relations.

\medskip
The proof of Proposition \ref{prop:charfunc} is complete now.
\end{proof}

\begin{proof}[Proof of Proposition \ref{prop:intlimit}]
To deduce Proposition \ref{prop:intlimit} from Proposition \ref{prop:charfunc},
we use that if $X$ is a real random variable with $\ee(X^{2k}) < \infty$,
induced distribution $\pp_X$ and characteristic function $f_X$, 
then, for any $T > 0$,
$$
\int_{[-T,+T]^c} x^{2k} \, \pp_X(dx) \leq \frac{T}{2} \int_{-2/T}^{+2/T} (-1)^k (f_X^{(2k)}(0)-f_X^{(2k)}(t)) \, dt \,.
$$
For the convenience of the reader, let us recall the argument:
Using that \linebreak $|\sin a / a| \leq 1$ for $a \in \myreal$
and $|\sin a / a| \leq \tfrac12$ for $|a| \geq 2$, we get
\begin{align*}
\frac{T}{2} \int_{-2/T}^{+2/T} & (-1)^k (f_X^{(2k)}(0)-f_X^{(2k)}(t)) \, dt \\
&= \frac{T}{2} \int_{-2/T}^{+2/T} \int_{-\infty}^{+\infty} x^{2k} (1 - e^{itx}) \, d\pp_X(x) \, dt \\ 
&= \int_{-\infty}^{+\infty} x^{2k} \frac{T}{2} \int_{-2/T}^{+2/T} (1 - e^{itx}) \, dt \, d\pp_X(x) \\ 
&= \int_{-\infty}^{+\infty} x^{2k} \left( 2 - 2 \sin(\tfrac{2x}{T}) / (\tfrac{2x}{T}) \right) \, d\pp_X(x) \\ 
&\geq \int_{[-T,+T]^c} x^{2k} \, d\pp_X(x) \,.
\end{align*}
Applying this inequality with $X = \overline{S}_n/\sqrt{n}$ and $T = C$,
we get 
\begin{align*}
      \int_{C}^{\infty} x^2 \, \overline{p}_n^*(x) \, dx 
&\leq \frac{C}{2} \int_{-2/C}^{+2/C} (-1) \left( f_{\overline{S}_n/\sqrt{n}}''(0) - f_{\overline{S}_n/\sqrt{n}}''(t) \right) \, dt \\
&\leq 2 \sup_{|t| \leq 2/C} \left| f_{\overline{S}_n/\sqrt{n}}''(0) - f_{\overline{S}_n/\sqrt{n}}''(t) \right| \,.
\end{align*}
Using Proposition \ref{prop:charfunc}, it follows that for any fixed $C > 0$,
we have
\begin{align*}
      \int_{C}^{\infty} x^2 \, \overline{p}_n^*(x) \, dx 
&\leq 2 \sup_{|t| \leq 2/C} \left| \hat\varphi_+''(0) - \hat\varphi_+''(t) \right| + o(1) \,.
\end{align*}
as $n \to \infty$.
Since $\hat\varphi_+''(t)$ is continuous at zero,
we may conclude that for $C = C(\varepsilon)$ sufficiently large,
we have
$$
\int_{C}^{\infty} x^2 \overline{p}_n^*(x) \, dx \leq \varepsilon
$$
for all sufficiently large $n \in \mynat$,
and the proof of Proposition \ref{prop:intlimit} is complete.
\end{proof}

\begin{remark*}
Let us outline another proof of Proposition \ref{prop:intlimit}.
This proof is shorter, \linebreak
but it is based on Spitzer's formula 
and the (classical) Erd\H{o}s-Kac theorem \eqref{thm:erdos-kac}.
Also, the preceding proof is more useful in that
it can be modified (under higher-order moment conditions)
to obtain more precise estimates on the rate of decay
in~Proposition \ref{prop:intlimit}.

\pagebreak[2]

We have to~show that the sequence of random variables
$((S_n^+/\sqrt{n})^2)_{n \in \mynat}$ is uniformly integrable.
It~follows from the classical Erd\H{o}s-Kac theorem \eqref{thm:erdos-kac}
that $(S_n^+/\sqrt{n})^2 \Rightarrow |Z|^2$ as~$n \to \infty$.
Now, it is well known \linebreak[2] that
for a family of integrable random variables 
$X,X_1,X_2,X_3,\dots$ with $X_n \Rightarrow X$,
\begin{align}
\label{eq:ui-criterion}
\text{$(X_n)_{n \in \mynat}$ is uniformly integrable if and only if $\ee|X_n| \to \ee|X|$,}
\end{align}
see \eg Lemma 4.11 in \cite{Kallenberg:2002}.
Thus, it remains to show that $\ee(S_n^+/\sqrt{n})^2 \to \ee|Z|^2$ $= 1$ as $n \to \infty$.

% \emph{I suspect that the following is well known,
% but I have not been able to find a~suitable reference so far.}

Our starting point is Spitzer's formula (see \eg \cite[p.\,618]{Feller-2}), 
which states that for $|s| < 1$ and $t \in \myreal$,
\begin{align}
  \sum_{n=0}^{\infty} s^n \ee(e^{it\overline{S}_n^{+}}) 
= \exp \left( \sum_{k=1}^{\infty} \frac{s^k}{k} \ee(e^{itS_k^{+}}) \right) \,.
\label{eq:spitzer-0}
\end{align}
Differentiating twice with respect to $t$ in \eqref{eq:spitzer-0},
we obtain
\begin{align*}
&\eskip \sum_{n=0}^{\infty} s^n \ee((\overline{S}_n^{+})^2 \, e^{it\overline{S}_n^{+}}) \\
&= \exp \left( \sum_{k=1}^{\infty} \frac{s^k}{k} \ee(e^{itS_k^{+}}) \right) \cdot \left[ \left( \sum_{k=1}^{\infty} \frac{s^k}{k} \ee(S_k^{+} \, e^{itS_k^{+}}) \right)^2 + \sum_{k=1}^{\infty} \frac{s^k}{k} \ee((S_k^{+})^2 \, e^{itS_k^{+}}) \right] \\
&= \sum_{n=0}^{\infty} s^n \ee(e^{it\overline{S}_n^{+}}) \cdot \left[ \left( \sum_{k=1}^{\infty} \frac{s^k}{k} \ee(S_k^{+} \, e^{itS_k^{+}}) \right)^2 + \sum_{k=1}^{\infty} \frac{s^k}{k} \ee((S_k^{+})^2 \, e^{itS_k^{+}}) \right]
\end{align*}
and therefore, setting $t = 0$,
\begin{align*}
   \sum_{n=0}^{\infty} s^n \ee((\overline{S}_n^{+})^2)
&= \sum_{n=0}^{\infty} s^n \cdot \left[ \left( \sum_{k=1}^{\infty} \frac{s^k}{k} \ee(S_k^{+}) \right)^2 + \sum_{k=1}^{\infty} \frac{s^k}{k} \ee((S_k^{+})^2) \right] \,.
\end{align*}
Hence, comparing coefficients, we may conclude that for any $n \geq 1$,
$$
\ee((\overline{S}_n^+)^2) = \sum_{\substack{k \geq 1,l \geq 1\\k+l \leq n}} \frac{1}{k} \ee(S_k^+) \frac{1}{l} \ee(S_l^+) + \sum_{1 \leq k \leq n} \frac{1}{k} \ee((S_k^+)^2) \,.
$$
Now, using the central limit theorem,
the fact that $\ee(S_n/\sqrt{n})^2 \to \ee Z^2 = 1$
and the criterion \eqref{eq:ui-criterion}, 
it is easy to see that
$$
\ee(S_n^+/\sqrt{n}) \to \ee(Z^+) = \frac{1}{\sqrt{2\pi}}
\quad\text{and}\quad
\ee((S_n^+/\sqrt{n})^2) \to \ee((Z^+)^2) = \frac{1}{2}
$$
as $n \to \infty$.
Therefore,
\begin{align*}
   \ee((\overline{S}_n^+/\sqrt{n})^2) 
&= \frac{1}{n} \sum_{m=1}^{n} \sum_{k=1}^{m-1} \frac{1}{k} \ee(S_k^+) \frac{1}{m-k} \ee(S_{m-k}^+) + \frac{1}{n} \sum_{k=1}^{n} \frac{1}{k} \ee((S_k^+)^2) \\ 
&= \frac{1}{n} \sum_{m=1}^{n} \sum_{k=1}^{m-1} \frac{1}{\sqrt{2\pi k}} \frac{1}{\sqrt{2\pi (m-k)}} + \frac{1}{n} \sum_{k=1}^{n} \frac{1}{2} + o(1) \\ 
&= \frac{1}{n} \sum_{m=1}^{n} \frac{1}{2\pi} \int_{0}^{1} \frac{1}{\sqrt{x(1-x)}} \, dx + \frac{1}{n} \sum_{k=1}^{n} \frac{1}{2} + o(1) \\
&= 1 + o(1)
\end{align*}
as $n \to \infty$.
This completes the proof of Proposition~\ref{prop:intlimit}.
\qed
\end{remark*}

\pagebreak[2]
\medskip
\section{Proof of Proposition \ref{prop:local}}

% The names \ttp and \ttf are due to the fact that these functions were called 
% \tilde{\tilde{p}} and \tilde{\tilde{f}} in a previous version of these notes.
\def\ttp{\tilde{p}}
\def\ttf{\tilde{f}}

\begin{proof}[Proof of Proposition \ref{prop:local}]
Let $p = (1-\varrho) q_1 + \varrho q_2$ be as in \eqref{eq:decomposition-p0}, 
and let $g_1$ and $g_2$ be the Fourier transforms of $q_1$ and $q_2$,
respectively. Then
$$
f^k(t) = \sum_{j=0}^{k} \tbinom{k}{j} (1-\varrho)^{j} \varrho^{k-j} g_1^j(t) g_2^{k-j}(t) \,.
$$
For $k \geq 3$, put
$$
\ttp_k(x) := \sum_{j=3}^{k} \tbinom{k}{j} (1-\varrho)^{j} \varrho^{k-j} (q_1^{*j} * q_2^{*(k-j)})(x)
$$
and
$$
\ttf_k(t) := \sum_{j=3}^{k} \tbinom{k}{j} (1-\varrho)^{j} \varrho^{k-j} g_1^j(t) g_2^{k-j}(t) \,.
$$
Note that $\ttf_n(t)$ is the Fourier transform of $\ttp_n(x)$
and that $\ttp_n(x)$ can be recovered from $\ttf_n(t)$
\linebreak[2] by means of Fourier inversion.
This follows from the fact that $g_1 \in L^2$ 
(being the~Fourier transform of a bounded probability density)
and $g_2 \in L^\infty$ 
(being the~Fourier transform of a probability measure).
%
% Also, by a similar argument, we~have $\ttf_n' \in L^1$.
% (This is the reason why we start from $j=3$ 
% and not from $j=2$ \linebreak[2] in the above definitions.)
%

Using our moment assumptions and the fact that $\varrho < \tfrac12$,
it is easy to see for $k \geq 3$ and $t \in \myreal$,
$$
\left| \tfrac{d^j}{dt^j} \big( f^k(t/\sqrt{n}) - \ttf_k(t/\sqrt{n}) \big) \right| = \myo(n^{-j/2} \, 2^{-k}) \,, \qquad j=0,1,2\,.
$$
It therefore follows from \eqref{eq:fk-all} 
that for $3 \leq k \leq n$, $|t| \leq \gamma n^{1/2}$ and $j=0,1,2$,
\begin{align}
\label{eq:ttfk-all} \left|\tfrac{d^j}{dt^j}(\ttf_k(t/\sqrt{n}) - e^{-kt^2/2n})\right| &\leq \delta_k \, (k/n)^{j/2} \, e^{-kt^2/4n} + \myo(n^{-j/2} \, 2^{-k}) \,.
\end{align}
Furthermore, there exist a constant $C_0 > 0$ 
and a constant $\eta \in (0,1)$ such that 
\linebreak for $k \geq 3$ and $|t| \geq \gamma$,
\begin{align}
\label{eq:ttRL-0}
     |\ttf_k(t)|
&\leq \sum_{j=3}^{k} \tbinom{k}{j} (1-\varrho)^{j} \varrho^{k-j} \left|g_1^j(t) g_2^{k-j}(t)\right|
 \leq C_0 \eta^{k-2} |g_1(t)|^2 \,,
\\[+5pt]
      |\ttf_k'(t)|
&\leq \sum_{j=3}^{k} \tbinom{k}{j} (1-\varrho)^{j} \varrho^{k-j} \left|\tfrac{d}{dt} \left[g_1^j(t) g_2^{k-j}(t)\right]\right|
 \leq C_0 k \eta^{k-3} |g_1(t)|^2 \,.
\label{eq:ttRL-1}
\end{align}
This follows from the fact that $g_1$ and $g_2$ also satisfy \eqref{eq:RL} 
(possibly with some modified constant $\eta$)
and that $g_1'$ and $g_2'$ are bounded,
$q_1$ and $q_2$ being probability measures with finite moments.

Recalling \eqref{eq:decomposition-p} and \eqref{eq:qrdef}
and using the non-negative densities $\ttp_k$ introduced above, 
we may write
\begin{align}
\overline{q}_n^*(x) = \sqrt{n} \, \sum_{k=3}^{n} (\ttp_k \ast G_{n-k})(\sqrt{n} \, x) + r_n(x) \,,
\label{eq:decomposition-q}
\end{align}
where the remainder term $r_n(x)$ is given by
\begin{align}
r_n(x) := & \, \sqrt{n} \, \sum_{k=1}^{n} \tbinom{k}{1} (1-\varrho) \ \varrho^{k-1} (q_1 \ \ast q_2^{*(k-1)} \ast G_{n-k})(\sqrt{n} \, x) \nonumber\\
        + & \, \sqrt{n} \, \sum_{k=2}^{n} \tbinom{k}{2} (1-\varrho)^2 \varrho^{k-2} (q_1^{*2} \ast q_2^{*(k-2)} \ast G_{n-k})(\sqrt{n} \, x) \,.
\label{eq:rndef}
\end{align}
The functions $r_n$ are the signed densities 
occurring in Proposition \ref{prop:local}. It~is easy to see 
that $\|r_n\|_1 = \myo(1/\sqrt{n})$ and $\|r_n\|_\infty = \myo(1)$.
Indeed, because $q_1$ and $q_2$ are probability densities,
$q_1$ is bounded and the total variation norm of $G_{n}$ 
is of order $\myo(1/\sqrt{n})$, we~have
$$
\|q_1^{*j} \ast q_2^{*(k-j)} \ast G_{n-k}\|_1
\leq 
\frac{C_1}{\sqrt{n-k+1}} \qquad (j=1,2)
$$ 
and \mbox{\qquad\qquad}
$$
\|q_1^{*j} \ast q_2^{*(k-j)} \ast G_{n-k}\|_\infty
\leq 
\frac{C_1}{\sqrt{n-k+1}} \qquad (j=1,2) \,,
$$
so that the asserted properties of the densities $r_n$ 
follow from the estimate
$$
     \sum_{k=j}^{n} \frac{\tbinom{k}{j} (1-\varrho)^j \varrho^{k-j}}{\sqrt{n-k+1}}
\leq \sum_{k=j}^{n} \frac{k^j \varrho^{k-j}}{\sqrt{n-k+1}}
= \myo(n^{-1/2}) \qquad (j=1,2) \,.
$$
Observe that all the terms in the big sum in \eqref{eq:decomposition-q}
contain the ``factor'' $q_1^{*2}(\sqrt{n} \, x)$ 
and therefore have Fourier transforms in $L^1$. 
Hence, similarly as in \cite{Alesh:77},
using Fourier inversion and \eqref{eq:invtransform},
we obtain the representation, for $x > 0$,
\begin{align*}
   \overline{q}_n^*&(x) - \sqrt{\tfrac{2}{\pi}} \, e^{-x^2/2} - r_n(x) \\
&= \frac{1}{2\pi} \int_\myreal e^{-itx} \bigg( \ttf_n(t/\sqrt{n}) - e^{-t^2/2} \bigg) \, dt \\
&+ \frac{1}{2\pi} \lim_{R \to \infty} \int_{-R}^{+R} e^{-itx} \, \frac{it}{\sqrt{2\pi n}} \bigg( \sum_{k=3}^{n-1} e^{-kt^2/2n} \frac{1}{\sqrt{n-k}} - \int_0^n e^{-ut^2/2n} \frac{du}{\sqrt{n-u}} \bigg) \, dt \\
&+ \frac{1}{2\pi} \int_\myreal e^{-itx} \bigg( \sum_{k=3}^{n-1} \Big( \ttf_k(t/\sqrt{n}) - e^{-kt^2/2n} \Big) \overline\varphi_{n-k}(t/\sqrt{n}) \bigg) \, dt \\
&+ \frac{1}{2\pi} \int_\myreal e^{-itx} \bigg( \sum_{k=3}^{n-1} e^{-kt^2/2n} \Big( \overline\varphi_{n-k}(t/\sqrt{n}) - (-\overline{a}_{n-k}) \, it/\sqrt{n} \Big) \bigg) \, dt \\
&+ \frac{1}{2\pi} \int_\myreal e^{-itx} \bigg( \sum_{k=3}^{n-1} e^{-kt^2/2n} \Big( (-\overline{a}_{n-k}) - \frac{1}{\sqrt{2\pi(n\!-\!k)}} \Big) \, it/\sqrt{n} \bigg) \, dt \,.
\end{align*}
\linebreak Denote the integrals on the right-hand side by $I_1,\hdots,I_5$.
Note that all the~integrals implicitly depend on $n$ and $x$.
We will consider each of them separately. 

\smallskip

\emph{Convention:}
We always assume that $n \geq 4$ and 
$x \in (0,\infty)$ (part (a)) or $x \in (0,e^{-1})$ (part (b)).
$\myo$- and $o$-\linebreak[2]bounds hold uniformly in these regions
(unless otherwise mentioned), and they may depend on the constants 
$\gamma,\delta_1,\delta_2,\delta_3,\dots$ introduced in Section~5,
on the constants $C_0$ and $\eta$ in \eqref{eq:ttRL-0} and \eqref{eq:ttRL-1},
and on the $L^2$-norm of the function $g_1$.

\medskip
\subsection{The proof of part (a)} 
Throughout this subsection we assume that $n \geq 4$ and $x \in (0,\infty)$.
The proof is very similar to that of Theorem~1 in \cite{Alesh:77}.

\medskip
\myparagraph{On the Integral $\pmb{I_1}$}
Using integration by parts, we~get
\begin{align*}
|I_1| &= \frac1x \left| \int_\myreal e^{-itx} \frac{d}{dt} \Big[ \ttf_n(t/\sqrt{n}) - e^{-t^2/2} \Big] \, dt \right| \\
&\leq \frac1x \int_{(-\gamma\sqrt{n},\gamma\sqrt{n})} \left| \frac{d}{dt} \Big[ \ttf_n(t/\sqrt{n}) - e^{-t^2/2} \Big] \right| \, dt \\
&+    \frac1x \int_{(-\gamma\sqrt{n},\gamma\sqrt{n})^c} \left| \frac{d}{dt} \Big[ \ttf_n(t/\sqrt{n}) - e^{-t^2/2} \Big] \right| \, dt \,.
\end{align*}
By \eqref{eq:ttfk-all}, the first integral on the right
is of the order $\myo(\delta_n + 2^{-n}) = o(1)$.
Furthermore, by \eqref{eq:ttRL-1}, \eqref{eq:gtb-1} and the fact 
that $g_1 \in L^2$, the second integral on the~right 
is of the order
\begin{multline*}
  \myo\bigg( \int_{(-\gamma\sqrt{n},\gamma\sqrt{n})^c} \Big( n \eta^{n-3} | g_1(t/\sqrt{n}) |^2 (1/\sqrt{n}) + |t| e^{-t^2/2} \Big) \, dt \bigg) \\
% \myo\bigg( \int_{(-\gamma,+\gamma)^c} n\eta^{n-2} |g_1(u)|^2 du + nu e^{-nu^2/2} \, du \bigg)
= \myo(n \eta^{n-3} + e^{-n\gamma^2/2} )
= o(1) \,.
\end{multline*}
Thus, $I_1 = o(1/x)$.

\medskip
\myparagraph{On the Integral $\pmb{I_2}$}
By \cite[Equation (24)]{Alesh:77},
we have $I_2 = \myo(1/(\sqrt{n} x))$.

\medskip
\myparagraph{On the Integral $\pmb{I_3}$}
For $k = 3,\hdots,n-1$, let
$$
I_{3,k} := \int_\myreal e^{-itx} \Big( \ttf_k(t/\sqrt{n}) - e^{-kt^2/2n} \Big) \overline\varphi_{n-k}(t/\sqrt{n}) \, dt \,.
$$
Then, similarly as in \cite{Alesh:77}, it follows via integration by parts that
\begin{multline*}
|I_{3,k}| = \frac{1}{x} \left| \int_\myreal e^{-itx} \frac{d}{dt} \bigg[ \Big( \ttf_k(t/\sqrt{n}) - e^{-kt^2/2n} \Big) \overline\varphi_{n-k}(t/\sqrt{n}) \bigg] \, dt \right| \\[+3pt] 
\leq x^{-1} |I_{3,k,1}| + x^{-1} |I_{3,k,2}| \,,
\end{multline*}
where $I_{3,k,1}$ and $I_{3,k,2}$ denote the integrals over the sets
$(-\gamma\sqrt{n},\gamma\sqrt{n})$ and \linebreak $(-\gamma\sqrt{n},\gamma\sqrt{n})^c$,
respectively.
It follows from \eqref{eq:ttfk-all}, \eqref{eq:phi-1} -- \eqref{eq:phi-3},
\eqref{eq:fzero} and \eqref{eq:ab} that
\begin{align*}
     |I_{3,k,1}|
\leq &\int_{(-\gamma\sqrt{n},\gamma\sqrt{n})} \Big|\ttf_k(t/\sqrt{n}) - e^{-kt^2/2n}\Big| \Big|\overline\varphi_{n-k}'(t/\sqrt{n}) (1/\sqrt{n})\Big| \, dt \\
   + &\int_{(-\gamma\sqrt{n},\gamma\sqrt{n})} \Big|\tfrac{d}{dt}\Big[\ttf_k(t/\sqrt{n}) - e^{-kt^2/2n}\Big]\Big| \Big|\overline\varphi_{n-k}(t/\sqrt{n})\Big| \, dt \\
   = &\myo\left( \int_{(-\gamma\sqrt{n},\gamma\sqrt{n})} \bigg[ \frac{\delta_k \, e^{-kt^2/4n}}{\sqrt{n(n-k)}} + \frac{2^{-k}}{\sqrt{n(n-k)}} \bigg] \, dt \right) \\
   + &\myo\left( \int_{(-\gamma\sqrt{n},\gamma\sqrt{n})} \bigg[ \frac{\delta_k (k/n)^{1/2} |t| \, e^{-kt^2/4n}}{\sqrt{n(n-k)}} + \frac{2^{-k}}{\sqrt{n(n-k)}} \bigg] \, dt \right) \\
   = &\myo\left( \frac{\delta_k}{\sqrt{k(n-k)}} + \frac{2^{-k}}{\sqrt{n-k}} \right) \,.
\end{align*}
Also, using \eqref{eq:phi-1}, \eqref{eq:phi-3}, 
\eqref{eq:fzero}, \eqref{eq:ab},
\eqref{eq:ttRL-0} and \eqref{eq:ttRL-1}, 
the Gaussian tail estimates \mbox{\eqref{eq:gtb-0} -- \eqref{eq:gtb-2}} 
and the~fact that $g_1 \in L^2$, we~get
\begin{align*}
     |I_{3,k,2}|
\leq &\int_{(-\gamma\sqrt{n},\gamma\sqrt{n})^c} |\ttf_k(t/\sqrt{n})| |\overline\varphi_{n-k}'(t/\sqrt{n}) (1/\sqrt{n})| \, dt \\
   + &\int_{(-\gamma\sqrt{n},\gamma\sqrt{n})^c} |\ttf_k'(t/\sqrt{n}) (1/\sqrt{n})| |\overline\varphi_{n-k}(t/\sqrt{n})| \, dt \\
   + &\int_{(-\gamma\sqrt{n},\gamma\sqrt{n})^c} |e^{-kt^2/2n}| |\overline\varphi_{n-k}'(t/\sqrt{n}) (1/\sqrt{n})| \, dt \\
   + &\int_{(-\gamma\sqrt{n},\gamma\sqrt{n})^c} |e^{-kt^2/2n} (kt/n)| |\overline\varphi_{n-k}(t/\sqrt{n})| \, dt \\
   = &\myo\left(\frac{ \eta^{k-2} + k\eta^{k-3} + \tfrac{1}{k\gamma} e^{-k\gamma^2/2} + e^{-k\gamma^2/2} }{\sqrt{n-k}}\right) \,.
\end{align*}
Therefore,
\begin{align}
\label{eq:I30}
I_{3,k} = \myo\left( \frac{1}{x} \frac{\delta_k + \widetilde\eta^k}{\sqrt{k(n-k)}} \right) \,,
\end{align}
where $\widetilde\eta := \tfrac12 (1 + \max \{ \tfrac12, \eta, e^{-\gamma^2/2} \}) \in (0,1)$.
Hence, using \eqref{eq:betasum-2}, we get \mbox{$I_{3} = o(1/x)$}.

\pagebreak[2]

\medskip
\myparagraph{On the Integral $\pmb{I_4}$}
It follows from \cite[Equation (47)]{Alesh:77} that $I_4 = o(1/x)$.

For the convenience of the reader, let us briefly sketch the argument from \cite{Alesh:77}.
For $k = 3,\hdots,n-1$, let
$$
I_{4,k} := \int_\myreal e^{-itx} e^{-kt^2/2n} \Big( \overline\varphi_{n-k}(t/\sqrt{n}) - (-\overline{a}_{n-k}) \, it/\sqrt{n} \Big) \, dt \,.
$$
Using integration by parts, we get
\begin{multline*}
     |I_{4,k}| 
\leq \frac{1}{x} \int_{\myreal} e^{-kt^2/2n} \Big( \tfrac{k}{n}|t| \, \Big| \overline\varphi_{n-k}(t/\sqrt{n}) - (-\overline{a}_{n-k}) \, it/\sqrt{n} \Big| \\ + \Big| \overline\varphi'_{n-k}(t/\sqrt{n}) (1/\sqrt{n}) - (-\overline{a}_{n-k}) \, i/\sqrt{n} \Big| \Big) \, dt \,.
\end{multline*}
We now split the integral at $\pm A$ ($A > 2$) 
and use 
the bounds \eqref{eq:phi-4} and \eqref{eq:phi-5}
in the region $(-A,+A)$ 
and 
the bounds \eqref{eq:phi-2} and \eqref{eq:phi-3}
in the region $(-A,+A)^c$.
In~combination with the Gaussian tail estimates \eqref{eq:gtb-0} and \eqref{eq:gtb-2}, 
we obtain
\begin{align*}
&\eskip \int_{\myreal} e^{-kt^2/2n} \Big( \tfrac{k}{n}|t| \, \Big| \overline\varphi_{n-k}(t/\sqrt{n}) - (-\overline{a}_{n-k}) \, it/\sqrt{n} \Big| \\&\qquad \,+\, \Big| \overline\varphi'_{n-k}(t/\sqrt{n}) (1/\sqrt{n}) - (-\overline{a}_{n-k}) \, i/\sqrt{n} \Big| \Big) \, dt \\
&\leq \int_{(-A,A)} |\overline{b}_{n-k}| e^{-kt^2/2n} \Big( \tfrac{k}{n} |t| \, t^2/n + |t|/n \Big) \, dt 
    \\&\qquad \,+\, \int_{(-A,A)^c} 2 |\overline{a}_{n-k}| e^{-kt^2/2n} \Big( \tfrac{k}{n} |t| \, |t|/\sqrt{n} + 1/\sqrt{n} \Big) \, dt \\
&   = \myo \left( A \frac{|\overline{b}_{n-k}|}{\sqrt{n} \sqrt{k}} \right) \\&\qquad \,+\, \myo \left( \frac{|\overline{a}_{n-k}|}{\sqrt{n}} \bigg[ \sqrt{\frac{n\pi}{2k}} \wedge \frac{n}{kA} e^{-kA^2/2n} + \sqrt{\frac{n\pi}{2k}} \wedge \frac{n}{kA} \left( \frac{k}{n} A^2 + 1 \right) e^{-kA^2/2n} \bigg] \right) \,,
\end{align*}
with implicit constants not depending on $n$ or $A$.
Note that the term in the square brackets is bounded 
by $\sqrt{2\pi n/k}$ for $k \leq n/A$ and 
by $(A+2) e^{-A/2}$ for $k \geq n/A$. 
Thus, using \eqref{eq:ab}, it follows that
\begin{align*}
   |I_4| 
&= \myo\left(\frac{A}{x} \sum_{k=3}^{n-1} \frac{|\overline{b}_{n-k}|}{\sqrt{n} \sqrt{k}}\right) \\&\qquad\,+\, \myo\left(\frac{1}{x} \sum_{k \leq n/A} \frac{1}{\sqrt{k(n-k)}} \right)  + \myo\left( \frac{(A+2) e^{-A/2}}{x} \sum_{k \geq n/A} \frac{1}{\sqrt{n(n-k)}} \right) \\
&= o(x^{-1} \, A) + \myo(x^{-1} \, A^{-1/2}) + \myo(x^{-1} \, (A+2) e^{-A/2}) \,.
\end{align*}
Letting $A \equiv A_n \to \infty$ sufficiently slowly as $n \to \infty$, 
we conclude that $|I_4| = o(1/x)$. 

\medskip
\myparagraph{On the Integral $\pmb{I_5}$}
It is shown in \cite[Equation (48)]{Alesh:77} that $I_5 = o(1/x)$.

\medskip

Clearly, combining the estimates for $I_1,\hdots,I_5$,
we get part (a) of Proposition~\ref{prop:local}.

\pagebreak[4]  % DANGER

\medskip
\subsection{The proof of part (b)} 
Throughout this subsection we assume that $n \geq 4$ and $x \in (0,e^{-1})$.
For these values of $x$, we can obtain somewhat better estimates
by avoiding the integration-by-parts step.

\medskip
\myparagraph{On the Integral $\pmb{I_1}$}
We have
\begin{align*}
|I_1| &= \left| \int_\myreal e^{-itx} \Big[ \ttf_n(t/\sqrt{n}) - e^{-t^2/2} \Big] \, dt \right| \\
&\leq \int_{(-\gamma\sqrt{n},\gamma\sqrt{n})} \Big| \ttf_n(t/\sqrt{n}) - e^{-t^2/2} \Big| \, dt \\
&+    \int_{(-\gamma\sqrt{n},\gamma\sqrt{n})^c} \Big| \ttf_n(t/\sqrt{n}) - e^{-t^2/2} \Big| \, dt \,.
\end{align*}
By \eqref{eq:ttfk-all}, the first integral on the right 
is of the order $\myo(\delta_n + \sqrt{n} \, 2^{-n}) = o(1)$.
Furthermore, by \eqref{eq:ttRL-0}, \eqref{eq:gtb-0} and the fact
that $g_1 \in L^2$, the second integral on the~right 
is of the order
\begin{multline*}
  \myo\left( \int_{(-\gamma\sqrt{n},\gamma\sqrt{n})^c} \Big( \eta^{n-2} \, |g_1(t/\sqrt{n})|^2 + e^{-t^2/2} \Big) \, dt \right) \\
= \myo(\sqrt{n} \eta^{n-2} + \tfrac{1}{\sqrt{n}} e^{-n\gamma^2/2} )
= o(1) \,.
\end{multline*}
Thus, $I_1 = o(1)$.

\medskip
\myparagraph{On the Integral $\pmb{I_2}$}
We have already mentioned that $I_2 = \myo(1/(\sqrt{n} x))$.
Now, using \eqref{eq:GI} and \eqref{eq:invtransform},
we also have
\begin{align*}
|I_2|
&=
\left| \lim_{R \to \infty} \int_{-R}^{+R} e^{-itx} \, \frac{it}{\sqrt{2\pi n}} \bigg( \sum_{k=3}^{n-1} e^{-kt^2/2n} \frac{1}{\sqrt{n-k}} - \int_0^n e^{-ut^2/2n} \frac{du}{\sqrt{n-u}} \bigg) \, dt \right| \\
& \leq \sum_{k=3}^{n-1} \int_{-\infty}^{+\infty} \frac{|t|}{\sqrt{2\pi n}} e^{-kt^2/2n} \frac{1}{\sqrt{n-k}} \, dt \\
     &\qquad\qquad\qquad\qquad\qquad\quad\ + \left| \lim_{R \to \infty} \int_{-R}^{+R} e^{-itx} \frac{it}{\sqrt{2\pi n}} \int_0^n e^{-ut^2/2n} \frac{du}{\sqrt{n-u}} \, dt \right| \\
&= \myo\left( \sum_{k=3}^{n-1} \frac{n}{k} \frac{1}{\sqrt{n(n-k)}} \right) + 2\pi \varphi(x) \\
&= \myo\Bigg(\sum_{1 \leq k \leq n/2} \frac{1}{k}\Bigg) + \myo\Bigg(\sum_{n/2 \leq k \leq n-1} \frac{1}{\sqrt{n(n-k)}}\Bigg) + \myo(1)
 = \myo(\log n) \,.
\end{align*}
Thus, $I_2 = \myo((\log n) \,\wedge\, (1/(\sqrt{n}x)))$.

\medskip
\myparagraph{On the Integral $\pmb{I_3}$}
For $k = 3,\hdots,n-1$, we can estimate the integral
$$
I_{3,k} := \int_\myreal e^{-itx} \Big( \ttf_k(t/\sqrt{n}) - e^{-kt^2/2n} \Big) \overline\varphi_{n-k}(t/\sqrt{n}) \, dt \,.
$$
in two different ways.

On the one hand, using integration by parts,
we obtain 
\begin{align}
\label{eq:I31}
I_{3,k} = \myo\left( \frac{1}{x} \frac{1}{\sqrt{k(n-k)}} \right) \,,
\end{align}
see \eqref{eq:I30}.

On the other hand, similar estimates (without integration by parts) 
yield
\begin{align*}
      |I_{3,k}|
&\leq \int_\myreal \Big| \ttf_k(t/\sqrt{n}) - e^{-kt^2/2n} \Big| \Big| \overline\varphi_{n-k}(t/\sqrt{n}) \Big| \, dt \\
&= \myo\left(\int_{(-\gamma\sqrt{n},\gamma\sqrt{n})} \bigg[ \frac{\delta_k \, |t| \, e^{-kt^2/4n}}{\sqrt{n(n-k)}} + \frac{2^{-k}}{\sqrt{n-k}} \bigg] \, dt \right) \\
&\qquad + \myo\left(\int_{(-\gamma\sqrt{n},\gamma\sqrt{n})^c} \Big| \ttf_k(t/\sqrt{n}) \Big| \Big| \overline\varphi_{n-k}(t/\sqrt{n}) \Big| \, dt \right) \\
&\qquad + \myo\left(\int_{(-\gamma\sqrt{n},\gamma\sqrt{n})^c} \Big| e^{-kt^2/2n} \Big| \Big| \overline\varphi_{n-k}(t/\sqrt{n}) \Big| \, dt \right) \\
&= \myo\left(\frac{n}{k} \frac{\delta_k}{\sqrt{n(n-k)}} + 2^{-k} \frac{\sqrt{n}}{\sqrt{n-k}} \right) \\
&\qquad + \myo\left( \eta^{k-2} \frac{\sqrt{n}}{\sqrt{n-k}} \right) + \myo\left( \frac{1}{k\gamma} e^{-k\gamma^2/2} \frac{\sqrt{n}}{\sqrt{n-k}} \right) \,,
\end{align*}
whence
\begin{align}
\label{eq:I32}
I_{3,k} = \myo\left( \frac{n}{k} \frac{1}{\sqrt{n(n-k)}} \right) \,.
\end{align}
Using \eqref{eq:I31} for $k \leq nx^2$ and \eqref{eq:I32} for $k \geq nx^2$
and recalling that $x \in (0,e^{-1})$, it~follows~that
\begin{multline*}
I_3 = \myo\left(\frac{1}{\sqrt{n} x} \sum_{1 \leq k \leq nx^2} \frac{1}{\sqrt{k}} + \sum_{nx^2 \leq k \leq n/2} \frac{1}{k} + \sum_{n/2 \leq k \leq n-1} \frac{1}{\sqrt{n(n-k)}}  \right) \\
    = \myo(1) + \myo(- \log x) + \myo(1)
    = \myo(- \log x) \,.
\end{multline*}
Thus, $I_3 = \myo(- \log x)$.

\medskip
\myparagraph{On the Integral $\pmb{I_4}$}
For $k=3,\hdots,n-1$, we can estimate the integral
$$
I_{4,k} := \int_\myreal e^{-itx} e^{-kt^2/2n} \Big( \overline\varphi_{n-k}(t/\sqrt{n}) - (-\overline{a}_{n-k}) \, it/\sqrt{n} \Big) \, dt \\
$$
in two different ways.
On the one hand, using integration by parts 
and \eqref{eq:phi-2}, \eqref{eq:phi-3} and \eqref{eq:ab},
we~have
$$
|I_{4,k}| \leq \frac{1}{x} \int_\myreal 2 |\overline{a}_{n-k}| e^{-kt^2/2n} \Big( \tfrac{k}{n} |t| \, |t|/\sqrt{n} + 1/\sqrt{n} \Big) \, dt = \myo \left( \frac{1}{x \sqrt{k(n-k)}} \right) \,.
$$
On the other hand, also using \eqref{eq:phi-2} and \eqref{eq:ab} 
(but without integration by parts), we~have
$$
|I_{4,k}| \leq \int_\myreal 2 |\overline{a}_{n-k}| e^{-kt^2/2n} \Big( |t|/\sqrt{n} \Big) \, dt = \myo \Bigg( \frac{n}{k} \frac{1}{\sqrt{n(n-k)}} \Bigg) \,.
$$
Thus, the same argument as for $I_3$ leads to the conclusion
that $I_4 = \myo(- \log x)$.

\medskip
\myparagraph{On the Integral $\pmb{I_5}$}
For $k=3,\hdots,n-1$, we can estimate the integral
$$
I_{5,k} := \int_\myreal e^{-itx} \, e^{-kt^2/2n} \Big( (-\overline{a}_{n-k}) - \frac{1}{\sqrt{2\pi(n\!-\!k)}} \Big) \, it/\sqrt{n} \, dt
$$
in two different ways.
On the one hand, using integration by parts and \eqref{eq:ab}, we~get
$$
  |I_{5,k}| 
= \myo \left( \frac{1}{x} \int_\myreal e^{-kt^2/2n} \Big( \frac{\frac{k}{n} |t| \, |t|}{\sqrt{n(n-k)}} + \frac{1}{\sqrt{n(n-k)}}  \Big) \, dt \right)
= \myo \left( \frac{1}{x} \frac{1}{\sqrt{k(n-k)}} \right) \,.
$$
On the other hand, using \eqref{eq:ab} (but without integration by parts), we get
$$
  |I_{5,k}| 
= \myo \left( \int_\myreal e^{-kt^2/2n} \Big( \frac{|t|}{\sqrt{n(n-k)}}  \Big) \, dt \right)
= \myo \left( \frac{n}{k} \frac{1}{\sqrt{n(n-k)}} \right) \,.
$$
Thus, the same argument as for $I_3$ leads to the conclusion
that $I_5 = \myo(- \log x)$.

\pagebreak[2]
\medskip

The proof of part (b) of Proposition \ref{prop:local} 
is completed by combining the previous estimates.
\end{proof}

\pagebreak[2]
\medskip
\section{Proof of Necessity in Theorem \ref{thm:ECLT}}

\begin{proof}[Proof of Necessity in Theorem \ref{thm:ECLT}]
Let us quote some well-known results from the literature: 
Suppose that $|s| < 1$.
By Spitzer's formula (see \eg \cite[p.\,618]{Feller-2}), 
we~have
\begin{align}
   \sum_{n=0}^{\infty} s^n \ee(e^{it\overline{S}_n^{+}}) 
%= \exp \left( \sum_{k=1}^{\infty} \frac{s^k}{k} \ee(e^{itS_k^{+}}) \right) \nonumber\\
&= \frac{1}{1-s} \exp \left( \sum_{k=1}^{\infty} \frac{s^k}{k} \int_0^\infty (e^{itx}-1) \, dF_k(x) \right)
\label{eq:spitzer-1}
\end{align}
for any $t \in \myreal$.
Also (see \eg \cite[p.\,416]{Feller-2}), we have
\begin{align*}
   1 + \sum_{n=1}^{\infty} s^n \pp(\overline{S}_n < 0)
= \exp \left( \sum_{k=1}^{\infty} \frac{s^k}{k} \pp(S_k < 0) \right)
= \frac{1}{1-s} \exp \left( - \sum_{k=1}^{\infty} \frac{s^k}{k} \pp(S_k \geq 0) \right) \,.
\end{align*}
Thus, Spitzer's formula \eqref{eq:spitzer-1}
can be rewritten as
\begin{align}
   \sum_{n=0}^{\infty} s^n \ee(e^{it\overline{S}_n^{+}}) 
&= \left( 1 + \sum_{n=1}^{\infty} s^n \pp(\overline{S}_n < 0) \right) \exp \left( \sum_{k=1}^{\infty} \frac{s^k}{k} \int_{[0,\infty)} e^{itx} \, dF_k(x) \right)
\label{eq:spitzer-2}
\end{align}
for any $t \in \myreal$.

Let us note that the preceding results hold 
without any assumptions on moments or on densities.
However, if the moment assumptions 
stated at the beginning of the introduction
are satisfied, then
\begin{align}
\label{eq:anbn-5}
\pp(\overline{S}_n < 0) = \Theta(n^{-1/2}) 
\end{align}
for $n \geq 1$ (see \eg \cite[pp.\,414f]{Feller-2}).
Indeed, more precise information is available.

Expanding the right-hand side of Spitzer's formula \eqref{eq:spitzer-2} 
into a power series in~$s$ and comparing coefficients, 
we~find that for any $n \geq 1$, 
\begin{multline*}
\ee(e^{it\overline{S}_n^{+}}) = \overline{F}_n(0) + \overline{F}_{n-1}(0) \int_0^\infty e^{itx} p_{1,+}(x) \, dx \\
+ \sum_{m=2}^{n} \overline{F}_{n-m}(0) \sum_{l=1}^{\infty} \sum_{\substack{k_1,\hdots,k_l \geq 1: \\ k_1+\hdots+k_l = m}} \frac{1}{l!} \frac{1}{k_1 \cdots k_l} \int_{0}^{\infty} e^{itx} (p_{k_1,+} \ast \hdots \ast p_{k_l,+})(x) \, dx \,,
\end{multline*}
where $\overline{F}_0(0) := 1$ and, for any $k \geq 1$, $p_{k,+}(x) := p_k(x)$ for $x > 0$ and $p_{k,+}(x) := 0$ for $x \leq 0$.
Hence, by the~unique\-ness theorem for Fourier transforms, we have
\begin{align}
\overline{p}_{n}(x) = \overline{F}_{n-1}(0) \, p_{1}(x) + \tilde{p}_n(x)
\label{eq:opn-p1}
\end{align}
for almost all $x > 0$, where $\tilde{p}_n$ is a certain subprobability density
on the positive half-line.

Now suppose that \eqref{eq:ECLT-1} holds.
Then, using Lemma \ref{lemma:scalar}, 
we have $D(\overline{p}_{n}^* \,|\, \varphi_+) < \infty$
for all sufficiently large $n \in \mynat$. 
It is easy to see that this implies 
$D(\overline{p}_{n} \,|\, \varphi_+) < \infty$
for all sufficiently large $n \in \mynat$.
Therefore, using \eqref{eq:opn-p1}, \eqref{eq:anbn-5} 
and the remark \eqref{eq:finiteness}
below Lemma~\ref{lemma:convexity},
we~may conclude that $D(p \,|\, \varphi_+) < \infty$,
which entails \eqref{eq:ECLT-2} by Lemma~\ref{lemma:scalar}.
\end{proof}

\pagebreak[2]
\medskip
\section{Proof of Theorem \ref{thm:totalvariation}}

\begin{proof}[Proof of Theorem \ref{thm:totalvariation}]
Fix $\varepsilon \in (0,1)$, and let $c \in (0,1)$ and $C \in (1,\infty)$
be such~that
\begin{align}
\label{eq:ccc-6}
\int_c^C \varphi_+(x) \, dx > 1 - \varepsilon \,.
\end{align}
Then, using Lemma \ref{lemma:entropy-q}\,(a) and Proposition \ref{prop:local}\,(a),
we have
\begin{align}
\label{eq:ccc-7}
     \int_{c}^{C} \left| \overline{p}_n^* - \varphi_+ \right| \, dx
\leq \int_{c}^{C} \left| \overline{p}_n^* - \overline{q}_n^* \right| \, dx
   + \int_{c}^{C} \left| \overline{q}_n^* - \varphi_+ \right| \, dx
   = o(1)
\end{align}
as $n \to \infty$,
which implies that
\begin{align}
\label{eq:ccc-8}
\int_c^C \overline{p}_n^*(x) \, dx > 1 - \varepsilon
\end{align}
for all sufficiently large $n \in \mynat$.
It follows from \eqref{eq:ccc-6} -- \eqref{eq:ccc-8} 
% and the fact that $\varphi_+$ and $\overline{p}_n^*$ are probability measures
that
$$
     d_{TV}(\overline{S}_n/\sqrt{n},|Z|) 
\leq \int_{\myreal} |\overline{p}_n^* - \varphi_+| \, dx
\leq \int_{(c,C)} |\overline{p}_n^* - \varphi_+| \, dx
+ \int_{(c,C)^c} (\overline{p}_n^* + \varphi_{+}) \, dx
< 2 \varepsilon
$$
for all sufficiently large $n \in \mynat$.
Since $\varepsilon \in (0,1)$ is arbitrary, 
Theorem \ref{thm:totalvariation} is~proved.
\end{proof}

\medskip

\renewcommand{\refname}{References}

\end{document}